\documentclass[12pt]{amsart}

\usepackage{amssymb, amsmath, amsthm}
\usepackage[margin=1in]{geometry}
\usepackage[dvipdfmx]{graphicx,xcolor} %For driver
\usepackage{tikz}

\allowdisplaybreaks

\numberwithin{equation}{section}

\theoremstyle{plain}
\newtheorem{thm}{Theorem}[section]

\newtheorem{lem}[thm]{Lemma}

\newtheorem*{referthmA}{Theorem A}

\newtheorem*{rem*}{Remark}

\theoremstyle{definition}
\newtheorem{defn}[thm]{Definition}
\newtheorem{rem}[thm]{Remark}

\newcommand{\N}{\mathbb{N}}
\newcommand{\R}{\mathbb{R}}

\newcommand{\Z}{\mathbb{Z}}
\newcommand{\ZnN}{(\mathbb{Z}^n)^N}

\newcommand{\calA}{\mathcal{A}}

\newcommand{\calF}{\mathcal{F}}

\newcommand{\calS}{\mathcal{S}}
\newcommand{\supp}{\mathrm{supp}\, }

\newcommand{\op}{\mathrm{Op}}

\newcommand{\veck}{\boldsymbol{k}}
\newcommand{\vecm}{\boldsymbol{m}}
\newcommand{\vecbeta}{\boldsymbol{\beta}}
\newcommand{\vecxi}{\boldsymbol{\xi}}
\newcommand{\veceta}{\boldsymbol{\eta}}
\newcommand{\vecnu}{\boldsymbol{\nu}}
\newcommand{\vecmu}{\boldsymbol{\mu}}

\newcommand{\vs}{\vspace{12pt}}

\begin{document}
\title[Multilinear pseudo-differential operators
of H\"ormander class]
{Boundedness  
of multilinear pseudo-differential operators 
with symbols in the 
H\"ormander class $S_{0,0}$} 

\author[T. Kato]{Tomoya Kato}
\author[A. Miyachi]{Akihiko Miyachi}
\author[N. Tomita]{Naohito Tomita}

\address[T. Kato]
{Division of Pure and Applied Science, 
Faculty of Science and Technology, Gunma University, 
Kiryu, Gunma 376-8515, Japan}

\address[A. Miyachi]
{Department of Mathematics, 
Tokyo Woman's Christian University, 
Zempukuji, Suginami-ku, Tokyo 167-8585, Japan}

\address[N. Tomita]
{Department of Mathematics, 
Graduate School of Science, Osaka University, 
Toyonaka, Osaka 560-0043, Japan}

\email[T. Kato]{t.katou@gunma-u.ac.jp}
\email[A. Miyachi]{miyachi@lab.twcu.ac.jp}
\email[N. Tomita]{tomita@math.sci.osaka-u.ac.jp}

\date{\today}

\thanks{This work was supported by JSPS KAKENHI, 
Grant Numbers 
20K14339 (Kato), 20H01815 (Miyachi), and 20K03700 (Tomita).}

%%%================================
\begin{abstract}
The multilinear pseudo-differential operators 
with symbols in the multilinear 
H\"ormander class $S_{0,0}$ are considered. 
A complete identification of the cases 
where those operators define 
bounded operators between local Hardy spaces 
is given. 
Some results for 
the boundedness 
between Wiener amalgam spaces are also given. 
These are extensions and improvements of 
the results known in the bilinear case. 
\end{abstract}

%%%=====================================
\maketitle

%%%======================================
\section{Introduction}\label{Introduction}
%%=========================================
\subsection{Introduction and main result}
\label{first-subsection}
%%%======================================

Throughout this paper, we fix positive integers $n$ and $N$. 
We consider $N$-linear pseudo-differential operators 
acting on functions on $\R^n$.

If $\sigma = \sigma (x, \xi_1, \dots, \xi_N)$, 
$(x, \xi_1, \dots, \xi_N)\in (\R^n)^{N+1}$, 
is a measurable function on $(\R^n)^{N+1}$ 
satisfying the estimate 
$\left| \sigma (x, \xi_1, \dots, \xi_N)  \right|
\le c (1+ |\xi_1|+ \cdots + |\xi_N|)^{L}$ with some $c$ and 
$L\in \R$, 
then the {\it $N$-linear pseudo-differential operator} $T_{\sigma}$ 
is defined by  
\[
T_{\sigma}(f_1,\dots,f_N)(x)
=\frac{1}{(2\pi)^{Nn}}
\int_{(\R^n)^N}e^{i x \cdot(\xi_1+\cdots+\xi_N)}
\sigma (x, \xi_1, \dots, \xi_N) 
\prod_{j=1}^N\widehat{f_j}(\xi_j)
\, d\xi_1 \dots d\xi_N,  
\]
where $x\in \R^n$, 
$f_1,\cdots,f_N \in \calS(\R^n)$, 
and $\widehat{f_j}$ denotes the Fourier transform.  
The function $\sigma$ is called the {\it symbol} of the operator 
$T_{\sigma}$. 
If 
$\sigma (x, \xi_1, \dots, \xi_N)$ does not depend on 
the variable $x$, then the function 
$\sigma = \sigma (\xi_1, \dots, \xi_N)$ is called the 
{\it multiplier} and 
$T_{\sigma}$ is called the 
{\it $N$-linear Fourier multiplier operator}.

We shall be interested in 
the boundedness of the $N$-linear pseudo-differential operators.   
For the boundedness of $T_{\sigma}$, 
we use the following terminology. 
Let $X_1,\dots, X_N$, and $Y$ be function spaces on $\R^n$ 
equipped with quasinorms 
$\|\cdot \|_{X_j}$ and $\|\cdot \|_{Y}$, 
respectively.  
If there exists a constant $C$ for which 
the inequality 
\begin{equation}\label{boundedness-XjY}
\|T_{\sigma}(f_1,\dots, f_N)\|_{Y}
\le C \prod_{j=1}^{N} 
\|f_j\|_{X_j} 
\end{equation}
holds for all 
$f_j\in \calS \cap X_j$, 
then, with a slight abuse of terminology, 
we say that 
$T_{\sigma}$ is bounded from 
$X_1 \times \cdots \times X_N$ to $Y$ 
and write 
$T_{\sigma}: X_1 \times \cdots \times X_N \to Y$. 
The smallest constant $C$ of 
\eqref{boundedness-XjY} 
is denoted by 
$\|T_{\sigma}\|_{X_1 \times \cdots X_N \to Y}$. 
If $\calA$ is a class of symbols,  
we denote by $\mathrm{Op}(\calA)$
the class of all operators $T_{\sigma}$ 
corresponding to $\sigma \in \calA$. 
If $T_{\sigma}: X_1 \times \cdots \times X_N \to Y$ 
for all $\sigma \in \calA$, 
then we write 
$\mathrm{Op}(\calA) 
\subset B (X_1 \times \cdots \times X_N \to Y)$.

Notice that 
$T_{\sigma}$ is originally defined 
for $f_j \in \calS (\R^n)$. 
If $T_{\sigma}: 
X_1 \times \cdots \times X_N \to Y$ 
in the sense given above, 
then, in many cases,  
we can employ some limiting argument 
to extend the definition of 
$T_{\sigma}$ to general $f_j \in X_j$ 
and prove that 
the inequality \eqref{boundedness-XjY} holds 
for all $f_j \in X_j$.

In this paper, we shall consider 
the boundedness of 
$T_{\sigma}$ from $h^{p_1} \times \cdots \times h^{p_N}$ 
to $h^p$, 
where $h^r$ denotes the local Hardy space 
of Goldberg \cite{goldberg 1979}. 
We shall also consider 
function spaces 
$L^p$ (the Lebesgue space), 
$H^p$ (the Hardy space), 
$BMO$, $bmo$ (the local BMO space), 
and the Wiener amalgam spaces $W^{p,q}_{s}$. 
The definitions of these function spaces will be given in 
Subsection \ref{functionspaces} 
and Section \ref{section2}.

It is known that if   
$p, p_1, \dots, p_N \in (0,\infty]$ and if 
for every  
$\sigma \in C_{0}^{\infty} ((\R^n)^N)$ 
the Fourier multiplier operator 
$T_{\sigma}$ is bounded 
from  
$h^{p_1}\times \cdots \times h^{p_N}$ to 
$h^p$ 
then 
$1/p \le 1/p_1 + \cdots + 1/p_N$;   
see \cite[Chapter 7, Proposition 7.3.7]{grafakos 2014m}, 
where the 
assertion is given for $L^p$ spaces but the 
proof can be modified to cover the case of $h^p$ spaces.    
Thus we shall consider 
the boundedness  
$T_{\sigma}: h^{p_1} \times \cdots \times h^{p_N} 
\to h^p$ only for 
$p, p_1, \dots, p_N \in (0,\infty]$ that satisfy 
$1/p \le 1/p_1 + \cdots + 1/p_N$.

We shall consider 
the class of symbols defined as follows.

%%========================
\begin{defn} 
For $m \in \R$, 
the class $S^{m}_{0,0}(\R^n, N)$ is defined to be 
the set of all 
$C^{\infty}$ functions 
$\sigma = \sigma (x, \xi_1, \dots, \xi_N)$ on 
$(\R^n)^{N+1}$ that satisfy the estimate 
\[
\left|
\partial^{\alpha}_x 
\partial^{\beta_1}_{\xi_1} \cdots \partial^{\beta_N}_{\xi_N} 
\sigma(x,\xi_1,\dots,\xi_N)
\right|
\le C_{\alpha,\beta_1,\dots,\beta_N}
\big( 1 + |\xi_1| +\cdots + |\xi_N| \big)^{m}
\]
for all multi-indices 
$\alpha, \beta_1, \dots, \beta_N \in \N_0^n
=\{ 0,1,2,\dots\}^{n}$. 
\end{defn}
%%====================================

The 
class $S^m_{0,0}(\R^n, N)$ is 
an $N$-linear version of 
the well-known 
H\"ormander class considered in the theory of 
linear pseudo-differential operators.  
The class $S^m_{0,0}(\R^n, N)$ for $N=2$ 
was considered by 
B\'enyi--Torres \cite{BT-2003, BT-2004} 
and investigated by 
B\'enyi--Maldonado--Naibo--Torres 
\cite{BMNT} 
and  
B\'enyi--Bernicot--Maldonado--Naibo--Torres 
\cite{BBMNT}. 
See these papers for the 
basic properties of the class 
$S^m_{0,0}(\R^n, 2)$ 
including 
symbolic calculus, duality, and interpolation.

We shall consider this problem: 
for $N\ge 2$ and 
for given $p, p_1, \dots, p_N \in (0, \infty]$,  
identify those $m\in \R$ 
such that 
$\op\left( S^{m}_{0,0}(\R^n, N) \right) \subset
B(h^{p_1} \times \cdots \times h^{p_N} \to h^{p})$. 
The origin of this problem may go back to 
B\'enyi-Torres \cite{BT-2004}, 
where the authors proved that 
for any $1 \le p_1, p_2, p < \infty$ satisfying 
$1/p = 1/p_1 + 1/p_2$
there exists a multiplier in $S^{0}_{0,0}(\R^n, 2)$ 
for which the corresponding bilinear 
Fourier multiplier operator is not bounded 
from $L^{p_1} \times L^{p_2}$ to $L^p$. 
Thus in order to have the relation 
$\op\left( S^{m}_{0,0}(\R^n, 2) \right) \subset
B(L^{p_1} \times L^{p_2} \to L^{p})$ 
the number $m$ must be negative.

In the case $N=2$, after the works of 
Michalowski--Rule--Staubach \cite{MRS-2014}
and 
B\'enyi--Bernicot--Maldonado--Naibo--Torres \cite{BBMNT}, 
the following theorem was proved 
by the second and the third named authors 
of the present paper.

%%============================================
\begin{referthmA}[\cite{MT-2013}]
Let 
$0 < p, \, p_1,\, p_2 \le \infty$, 
$1/p =1/p_1 + 1/p_2$, and $m\in \R$. 
Then the boundedness 
\begin{equation*}
\op\left( S^{m}_{0,0}(\R^n, 2) \right) \subset
B(h^{p_1} \times h^{p_2} \to h^{p}), 
\end{equation*}
where if $p_1, p_2$, or $p$ is equal to $\infty$  
then 
the corresponding $h^{p_1}$, $h^{p_2}$, or $h^{p}$ 
should be replaced by $bmo$, 
holds if and only if 
\begin{equation}\label{criticalm-bilinear}
m 
\le 
-n \left(\max\left\{
\frac{1}{\,2\,}, \,
\frac{1}{p_1}, \,
\frac{1}{p_2}, \,
1-\frac{1}{p_1} - 
\frac{1}{p_2} , \, 
\frac{1}{p_1} + 
\frac{1}{p_2} -\frac{1}{\,2\,}
\right\}\right). 
\end{equation}
\end{referthmA}
%%============================================

The purpose of the present paper is to generalize this theorem 
to the case $N\ge 2$ 
and remove the restriction 
$1/p=1/p_1 + 1/p_2$. 
Our result also gives an improvement of 
Theorem A 
in the case 
$p=\infty$;  
we prove that 
$\op\left( S^{-n}_{0,0}(\R^n, 2) \right) \subset
B(bmo \times bmo \to L^{\infty})$,   
whereas  
Theorem A gives the weaker result 
$\op\left( S^{-n}_{0,0}(\R^n, 2) \right) \subset
B(bmo \times bmo \to bmo)$.

The main result of the present paper reads as follows.

%%============================================
\begin{thm}\label{main-thm-hp}
Let $N\ge 2$, 
$0 < p, \, p_1,\, \dots, \, p_N \le \infty$, 
$1/p \le 1/p_1 + \dots + 1/p_N$, and $m\in \R$. 
Then the boundedness 
\begin{equation}\label{boundedness-hpj-hp}
\op\left( S^{m}_{0,0}(\R^n, N) \right) \subset
B(h^{p_1} \times \cdots \times h^{p_N} \to h^{p})
\end{equation}
holds if and only if 
\begin{equation}\label{criticalm}
m \le 
\min \left\{ \frac{n}{\,p\,},\, \frac{n}{\,2\,} \right\} 
- \sum_{j =1}^{N} 
\max \left\{ 
\frac{n}{\,p_j\,},\, \frac{n}{\,2\,}  \right\}.  
\end{equation}
If \eqref{criticalm} is satisfied 
and if some of the $p_j$'s are equal to $\infty$, 
then 
\eqref{boundedness-hpj-hp} holds 
with the corresponding $h^{p_j}$ replaced by $bmo$. 
\end{thm}
%%=========================================

Observe that if 
$N=2$ and $1/p=1/p_1 + 1/p_2$ 
then 
\eqref{criticalm} coincides with 
\eqref{criticalm-bilinear}.

%%==========================
\begin{rem} 
(1) In the case $N=1$, 
the operator $T_{\sigma}$ is 
a linear pseudo-differential operator 
and the class 
$S^{m}_{0,0}(\R^n, 1)$ is the 
usual H\"ormander class. 
For this case, the following assertion holds. 
{\it 
Let 
$0 < p_1\le p \le \infty$ and $m\in \R$. 
Then the boundedness 
\begin{equation*}
\op\left( S^{m}_{0,0}(\R^n, 1) \right) \subset
B(h^{p_1}  \to h^{p})  
\end{equation*}
with $h^{p}$ ($h^{p_1}$, resp.) replaced by $bmo$ 
when $p=\infty$ ($p_1=\infty$, resp.)  
holds if and only if \/} 
\begin{equation*}
m \le 
\min \left\{ \frac{n}{\,p\,},\, \frac{n}{\,2\,} \right\} 
- \max \left\{ 
\frac{n}{\,p_1\,},\, \frac{n}{\,2\,}  \right\}.   
\end{equation*}
In fact, this result is already known. 
The `if' part for $p_1=p$ is given 
by Calder\'on--Vaillancourt \cite{CV} (the case $p_1 = p_2=2$),  
by Fefferman \cite{F-1973} and 
Coifman--Meyer \cite{CM-Ast} (the case $1< p_1=p<\infty$), 
and by 
Miyachi \cite{Miyachi-1987} 
and P\"aiv\"arinta--Somersalo \cite{PS-1988}  (the case $0<p_1=p\le 1$).  
The `if' part for the case $p_1 < p$ can be deduced from the 
case $p_1=p$ with the aid of the mapping properties 
of the fractional integration operator 
and symbolic calculus in $S^{m}_{0,0}(\R^n, 1)$. 
The `only if' part for the case $p_1=p$ 
can be found, {\it e.g.\/}, in 
\cite[Section 5]{Miyachi-1980}.  
It should be remarked that the method of the present paper 
can be applied, with only slight modification, 
to the case $N=1$ to prove the above assertion.

(2) 
In the case $N=1$, 
the target space 
$h^{\infty}$ has to be replaced by $bmo$, 
in particular, 
the boundedness 
$\op (S^{-n/2}_{0,0}(\R^n, 1))\subset B (L^{\infty}\to L^{\infty})$ 
does not hold 
(see \cite[Section 5, p.\ 151]{Miyachi-1987}).  
However, if $N\ge 2$ then 
we have the boundedness with the target space 
$h^{\infty}=L^{\infty}$. 
This is related to the fact that 
Lemma \ref{productLweak1} to be given in Section 
\ref{section2} 
holds only for $N\ge 2$.  
\end{rem}
%%================================

%%=====================================
\subsection{Refined version of the main theorem}
%%======================================

In fact, we shall give a slightly refined version of 
Theorem \ref{main-thm-hp}. 
To give the refined version, 
we use the following.

%%===================================
\begin{defn}
For $\vecm = (m_1,\dots,m_N)\in\R^N$, 
the class 
$S^{\vecm}_{0,0}(\R^n, N)$ is 
defined to be the set of all 
$C^{\infty}$ functions 
$\sigma = \sigma (x, \xi_1, \dots, \xi_N)$ on 
$(\R^n)^{N+1}$ that satisfy the estimate 
\[
\left|
\partial^{\alpha}_x 
\partial^{\beta_1}_{\xi_1} \cdots \partial^{\beta_N}_{\xi_N} 
\sigma(x,\xi_1,\dots,\xi_N)
\right|
\le C_{\alpha,\beta_1,\dots,\beta_N}
\prod_{j=1}^{N}
\big( 1 + |\xi_j| \big)^{m_j}
\]
for all $\alpha, \beta_1, \dots, \beta_N \in \N_0^n$. 
For $m\in \R$, 
we write 
$S^{m}_{0,0}(\R^n, N)^{\times}$ 
to denote 
the set of all $\sigma \in S^{m}_{0,0}(\R^n, N)$ 
that do not depend on the variable $x$. 
We consider $\sigma \in S^{m}_{0,0}(\R^n, N)^{\times}$ 
as a function defined on $(\R^n )^N$.
\end{defn}
%%==============================

Notice that 
if 
$m_1, \dots, m_N \le 0$ and if 
$m_1+\dots+m_N = m$  
then 
$S^{m}_{0,0}(\R^n, N)
\subset 
S^{\vecm}_{0,0}(\R^n, N)$. 
Thus 
the assertion (1) of the following theorem 
is a refined version of the `if' part of 
Theorem \ref{main-thm-hp} for the case $p<\infty$. 
The assertion (2) of this theorem is essentially a restatement 
of the `if' part of 
Theorem \ref{main-thm-hp} for the case $p=\infty$.

%%============================================
\begin{thm}\label{main-thm-sufficiency}
Let $N\ge 2$, 
$0<p, \, p_1,\, \dots, \, p_N \le \infty$, and 
$1/p \le 1/p_1 + \dots + 1/p_N$. 
\begin{enumerate}
\item 
If $p<\infty$ and if ${\vecm} = (m_1, \dots, m_N)\in \R^{N}$ satisfies 
\[
- \max \left\{ 
\frac{n}{\,p_j\,},\, \frac{n}{\,2\,}  \right\} 
< m_j <
\frac{n}{\,2\,} -\max \left\{ 
\frac{n}{\,p_j\,},\, \frac{n}{\,2\,}  \right\},  
\quad j=1, \dots, N, 
\]
and 
\[
m_1 + \cdots +m_N = 
\min \left\{ \frac{n}{\,p\,},\, \frac{n}{\,2\,} \right\} 
- \sum_{j =1}^{N} 
\max \left\{ 
\frac{n}{\,p_j\,},\, \frac{n}{\,2\,}  \right\},  
\]
then 
$\op\left( S^{\vecm}_{0,0}(\R^n, N) \right) \subset
B(h^{p_1} \times \cdots \times h^{p_N} \to h^{p})$. 

\item 
If 
$m=
- \sum_{j =1}^{N} 
\max \big\{ 
\frac{n}{\,p_j\,},\, \frac{n}{\,2\,}  \big\}$,   
then 
$\op\left( S^{m}_{0,0}(\R^n, N) \right) \subset
B(h^{p_1} \times \cdots \times h^{p_N} \to L^{\infty})$. 
\end{enumerate}
In the above assertions, if some of the $p_j$'s are equal to $\infty$, 
then the conclusions  
hold with the corresponding $h^{p_j}$ replaced by $bmo$. 
\end{thm}
%%============================

The next theorem covers all $N\ge 1$. 
This theorem is a slightly refined version of 
the `only if' part of 
Theorem \ref{main-thm-hp} 
since 
$H^p \subset h^p$ and $L^{\infty} \subset bmo \subset BMO$.

%%===============
\begin{thm}\label{main-thm-necessity}
Let $N\ge 1$,  
$0 < p, p_1,\dots, p_N \le \infty$,  
$1/p \le 1/p_1 + \cdots + 1/p_N$, and  
$m\in \R$. 
If 
\begin{equation*}
\op ((S^{m}_{0,0}(\R^n, N) )^{\times}) 
\subset 
B(H^{p_1} \times \cdots \times H^{p_N} \to L^{p}), 
\end{equation*}
with $L^p$ is replaced by $BMO$ when $p=\infty$, 
then \eqref{criticalm} holds. 
\end{thm}
%%========================================

One of the main ideas of this paper is 
to use Wiener amalgam spaces. 
We first prove the estimate of 
pseudo-differential operators in 
Wiener amalgam spaces and then derive the 
estimate in $h^p$ spaces by using 
embedding relations between 
Wiener amalgam spaces and $h^p$. 
Similar method is also used 
in the papers \cite{KMT-arxiv} and \cite{KMT-arxiv-2} 
by the same authors.

%%=======================================
\subsection{Function spaces, etc.} 
\label{functionspaces}
%%%=====================================

We give the definitions and properties 
of some function spaces 
and give some notations.

$L^p(\R^n)$, $0<p\le \infty$, 
denotes the usual 
Lebesgue space on $\R^n$. 
$L^{p,q}(\R^n)$, $0<p< \infty$, $0<q \le \infty$, 
denotes the Lorentz space; 
see, {\it e.g.}, \cite[Chapter 1, Section 1.4]{grafakos 2014c}. 
In particular, 
the space $L^{p,\infty} (\R^n)$, $0<p<\infty$, 
consists of all measurable functions 
$f$ on $\R^n$ such that 
\[
\|f\|_{L^{p,\infty}}
=\sup_{0<\lambda<\infty} 
\lambda \left| 
\{
x \in \R^n 
\mid |f(x)| > \lambda 
\}
\right|^{1/p} 
<\infty, 
\]
where $|E|$ denotes the Lebesgue measure of $E\subset \R^n$. 
It holds that 
$L^{p,p}(\R^n)=L^{p}(\R^n)$ for $0<p<\infty$.

For $0<q<\infty$, 
$\ell^q (\Z^n)$ 
denotes the class 
of all  
complex sequences $a=\{a_k\}_{k\in \Z^n}$ 
such that 
$ \| a \|_{ \ell^q (\Z^n)} 
= 
\left( \sum_{ k \in \Z^n } 
| a_k |^q \right)^{ 1/q } <\infty$.

Let $\phi \in \calS(\R^n)$ be such that
$\int_{\R^n}\phi(x)\, dx \neq 0$ and 
let 
$\phi_t(x)=t^{-n}\phi(x/t)$ for $t>0$. 
The space $H^p= H^p(\R^n)$, $0<p\leq\infty$,  
consists of
all $f \in \calS'(\R^n)$ such that 
$\|f\|_{H^p}=\|\sup_{0<t<\infty}|\phi_t*f|\|_{L^p}
<\infty$. 
The space $h^p= h^p(\R^n)$, $0<p\leq\infty$,  
consists of
all $f \in \calS'(\R^n)$ such that 
$\|f\|_{h^p}=\|\sup_{0<t<1}|\phi_t*f|\|_{L^p}
<\infty$.
It is known that $H^p$ and $h^p$ 
do not depend on the choice of the function $\phi$ 
up to the equivalence of quasinorm.

Obviously $H^p \subset h^p$. 
If $1 < p \leq \infty$, 
then 
$H^p = h^p = L^p$ with equivalent norms. 
If $0 < p \le 1$, then 
the inequality 
$\|f\|_{L^p}\lesssim \|f\|_{h^p}$ 
holds for all 
$f \in h^{p}$ which are 
defined by locally  integrable 
functions on $\R^n$.

The space $BMO=BMO(\R^n)$ consists of
all locally integrable functions $f$ on $\R^n$ 
such that
\[
\|f\|_{BMO}
=\sup_{R}\frac{1}{|R|}
\int_R |f(x)- f_{R}|\, dx
<\infty,
\]
where $f_R=|R|^{-1}\int_R f(x) \, dx$ 
and $R$ ranges over all the cubes in $\R^n$.
The space $bmo=bmo(\R^n)$ consists of
all locally integrable functions $f$ on $\R^n$ 
such that
\[
\|f\|_{bmo}
=\sup_{|R| \le 1}\frac{1}{|R|}
\int_{R}|f(x)-f_R|\, dx
+\sup_{|R|\geq1}\frac{1}{|R|}
\int_R |f(x)|\, dx
<\infty,
\]
where $R$ denotes cubes in $\R^n$.

The embedding  
$L^{\infty} \subset bmo \subset BMO$ holds.

The definition of 
Wiener amalgam spaces 
together with 
their embedding properties 
will be given in Section \ref{section2}. 
\vs

The following notations are 
used throughout the paper.

For two nonnegative functions $A(x)$ and $B(x)$ defined 
on a set $X$, 
we write $A(x) \lesssim B(x)$ for $x\in X$ to mean that 
there exists a positive constant $C$ such that 
$A(x) \le CB(x)$ for all $x\in X$. 
We often omit to mention the set $X$ when it is 
obviously recognized.  
Also $A(x) \approx B(x)$ means that
$A(x) \lesssim B(x)$ and $B(x) \lesssim A(x)$.

$C_{0}^{\infty}(\R^d)$ denotes 
the set of all the $C^{\infty}$ functions on $\R^d$ 
that have compact supports.

The symbols $\calS (\R^d)$ and 
$\calS^\prime(\R^d)$ denote 
the Schwartz class of rapidly 
decreasing smooth functions
and 
the space of tempered distributions 
on $\R^d$, respectively. 
The Fourier transform and the inverse 
Fourier transform of $f \in \calS(\R^d)$ are defined by
\begin{align*}
\mathcal{F} f  (\xi) 
&= \widehat {f} (\xi) 
= \int_{\R^d}  e^{-i \xi \cdot x } f(x) \, dx, 
\\
\mathcal{F}^{-1} f (x) 
&
= \frac{1}{(2\pi)^d} \int_{\R^d}  e^{i x \cdot \xi } f( \xi ) \, d\xi,
\end{align*}
respectively. 
For $m \in \calS^\prime (\R^d)$, 
the linear Fourier multiplier operator $m(D)$ is defined by
\begin{equation*}
m(D) f 
=
\mathcal{F}^{-1} \left[m \cdot \mathcal{F} f \right].
\end{equation*}

For $1\le p\le \infty$, the conjugate index $p^{\prime}$ 
is defined by 
$1/p + 1/p^{\prime}=1$.

For $x \in \R^d$, 
we write 
$\langle x \rangle = (1 + |x|^2)^{1/2}$. 
\vs

The contents of the rest of the paper is as follows. 
In Section \ref{section2}, 
we recall some embedding relations 
between Wiener amalgam spaces, $h^p$, and $bmo$ 
and prove some 
inequalities that are used in the proof of Theorem \ref{main-thm-sufficiency}.  
In Sections \ref{section3} and \ref{section4}, 
we prove Theorems \ref{main-thm-sufficiency} and 
\ref{main-thm-necessity}, respectively. 
Notice that the main theorem, Theorem \ref{main-thm-hp}, 
directly follows from 
Theorems \ref{main-thm-sufficiency} and 
\ref{main-thm-necessity}.

%%%%%%%%%%%%%%%%%%%%%%%%%%%%%%%%%%%%
\section{Preliminaries for the proof of Theorem 
\ref{main-thm-sufficiency}} 
\label{section2}
%%%%%%%%%%%%%%%%%%%%%%%%%%%%%%%%%%%%

We begin with the definition of the Wiener amalgam spaces. 

%%===============
\begin{defn}\label{Def-WienerAmalgam}
Let $\kappa\in C_{0}^{\infty}(\R^n)$ 
be a function such that 
\begin{equation}\label{condofkappa}
\left| \sum _{k\in \mathbb{Z}^{n}} \kappa(\xi-k) \right| \geq 1, 
\quad \xi \in \R^n.  
\end{equation}
Then for $0< p,q \leq \infty$ and $s\in \R$, 
the {\it Wiener amalgam space} $W^{p,q}_{s}$ is defined to be the 
set of all $f \in \calS'(\R^n)$ such that the 
quasinorm 
\begin{equation*}
\|f\|_{W^{p,q}_{s}} 
= \Big\| \big\|
\langle k \rangle^{s} \, \kappa(D-k)f(x)
\big\|_{\ell^{q}_{k}(\Z^n)} \Big\|_{L^p_{x}(\R^n)}
\end{equation*}
is finite. 
If 
$s=0$, we write 
$W^{p,q} = W^{p,q}_{0}$. 
\end{defn}
%%===============

The basic definition of the Wiener amalgam space  
is due to 
Feichtinger \cite{Fei-1981} and Triebel \cite{Tri-1983}.  
The usual definition of the Wiener amalgam space 
is given by 
a function $\kappa$ that satisfies 
$\sum _{k\in \mathbb{Z}^{n}} \kappa(\cdot-k) \equiv 1$.
However, a bit more general $\kappa$ 
satisfying \eqref{condofkappa} 
will be convenient in our argument.  
The space $W^{p,q}_{s}$ 
does not depend on the choice of the function $\kappa$ 
up to the equivalence of quasinorm. 
The space $W^{p,q}_{s}$ is a quasi-Banach space 
(Banach space if $1 \leq  p,q \leq \infty$) and 
$\mathcal{S} \subset W^{p,q}_{s} \subset \mathcal{S}^\prime$. 
If $0 < p,q < \infty$, then $\mathcal{S}$ is dense in $W^{p,q}_{s}$.

Several quasinorms equivalent to $\|f\|_{W^{p,q}_{s}}$ are known. 
Firstly, 
$\|f\|_{W^{p,q}_{s}}$ is equivalent to 
its continuous version: 
\begin{equation}\label{continuous-version}
\|f\|_{W^{p,q}_{s}}
\approx 
\Big\| \big\|
\langle y \rangle^{s} \, \kappa(D-y)f(x)
\big\|_{L^{q}_{y}(\R^n)} \Big\|_{L^p_{x}(\R^n)}. 
\end{equation}
If we write $\varphi = \calF^{-1}\kappa$, then 
$\kappa(D-y)f(x) = \big\langle 
\varphi (x-z) e^{-i y\cdot z} , f(z) 
\big\rangle e^{i y \cdot x}$ and hence 
\eqref{continuous-version} for $s=0$ implies 
\begin{align}
\|f\|_{W^{p,q}}
&
\approx 
\left\| 
\left\| 
\big\langle 
\varphi (x-z) e^{-i y\cdot z} , f(z) 
\big\rangle 
\right\|_{L^q_y (\R^n)}
\right\|_{L^p_x (\R^n)}
\label{continuous-version-2}
\\
&
= 
\left\| 
\left\| 
\big\langle 
\kappa (-y+\xi ) e^{i \xi \cdot x} , \widehat{f}(\xi) 
\big\rangle 
\right\|_{L^q_y (\R^n)}
\right\|_{L^p_x (\R^n)}. 
\label{modulation-hatf}
\end{align}
The quasinorm 
on the right hand side of 
\eqref{continuous-version-2} is 
sometimes 
adopted for the definition of $W^{p,q}$. 
The quasinorm 
\eqref{modulation-hatf} 
is the quasinorm of $\widehat{f}$ 
in the {\it modulation space}. 
If $q=2$, then 
\eqref{continuous-version-2} 
combined with Plancherel's theorem 
yields 
\begin{equation*}
\|f\|_{W^{p,2}}
\approx 
\left\| 
\left\| 
\varphi (x-z) f(z) 
\right\|_{L^2_z (\R^n)}
\right\|_{L^p_x (\R^n)},   
\end{equation*}
the right hand side of which 
is the quasinorm of $f$ in the $L^2$-based 
{\it amalgam space}. 
For these facts, see 
\cite{Fei-1981, Fei-1983, 
Tri-1983, Gro-book-2001, 
GS-2004, Kob-2006, WH-2007, 
GCFZ-2019}.  
For modulation space, 
see 
\cite{Fei-1983, GS-2004, Gro-book-2001, Kob-2006, WH-2007}.  
For amalgam spaces, 
see \cite{fournier stewart 1985} and 
\cite{holland 1975}.

The following embedding relations are known. 

%%=======================================
\begin{lem}\label{Waembd}
The following embeddings hold: 
\begin{align}
&
W^{p_1,q_1} \hookrightarrow W^{p_2,q_2}, 
\quad 
0<p_1\le p_2 \le \infty, \;\; 0<q_1\le q_2 \le \infty; 
\label{emb-WW}
\\
&
L^{p} \hookrightarrow W^{p,p'}, 
\quad 1 \leq p \leq 2; 
\label{emb-LW-1}
\\
&
L^{p} \hookrightarrow W^{p,2}, 
\quad 
2 \leq p \le \infty; 
\label{emb-LW-2}
\\
&
h^p \hookrightarrow W^{p,2}_{n(1/2-1/p)}, 
\quad 0 < p \le 2; 
\label{emb-hW}
\\
&
bmo \hookrightarrow W^{\infty,2}; 
\label{emb-bmoW}
\\
&
W^{p,2} \hookrightarrow h^{p}, 
\quad 0 < p \leq 2; 
\label{emb-Wh}
\\
&
W^{p,p'} \hookrightarrow L^{p}, 
\quad 
2 \leq p \leq \infty.  
\label{emb-WL}
\end{align}
\end{lem}
%%======================================

The embeddings 
\eqref{emb-LW-1}, \eqref{emb-LW-2}, and 
\eqref{emb-WL} 
are given in \cite[Theorems 1.1 and 1.2]{CKS-2015}.  
The embeddings 
\eqref{emb-hW} and \eqref{emb-Wh} are 
given in \cite[Theorem 1.2]{GWYZ-2017}.  
For \eqref{emb-WW} and 
\eqref{emb-bmoW},  
see 
{\cite{CKS-2015, GCFZ-2019, GWYZ-2017, KMT-arxiv-2}.

%%======================================
\begin{lem}[{\cite[Lemma 2.2]{PS-1988}}]
\label{PSdec}
For each $M\in \N$, there exist 
functions $\{ \chi_{\ell} \}_{\ell \in \Z^n} \subset 
C_{0}^{\infty} (\R^n)$ such that 
\begin{align*}
&\supp \chi_{\ell} \subset [-1,1]^n, \qquad 
\sup_{\ell} \| \calF^{-1} \chi_\ell \|_{L^1} \lesssim 1, \\
&\sum_{\ell \in \Z^n} \langle \ell \rangle^{-2M} 
\chi_\ell (\zeta) \langle \zeta \rangle^{2M} = 1
\;\;\textrm{for all}\;\; \zeta \in \R^n.
\end{align*}
Here 
the implicit constant in 
$\lesssim$ depends only on 
$n$ and $M$. 
\end{lem}
%%%============================

%%%============================
\begin{lem}
\label{productLweakp'}
Let $N \ge 2$, $1<r<\infty$, 
and let $a_1, \dots, a_N$ be real numbers satisfying 
$-n/2< a_j < 0$ and 
$\sum_{j=1}^N a_j=n/r-Nn/2$. 
Then 
the 
following inequality 
holds for all nonnegative functions 
$A_1, \, \dots, \, A_N$ on $\Z^n$: 
\[
\left\|
\sum_{\substack{\nu_1, \, \dots, \, \nu_N \in \Z^n, \\
\nu_1+\dots+ \nu_N = \mu} }
\, 
\prod_{j=1}^N (1+ |\nu_j| )^{a_{j}} A_j(\nu_j)
\right\|_{\ell^{r'}_{\mu}(\Z^n)}
\lesssim \prod_{j=1}^N
\|A_j\|_{\ell^2 (\Z^n) }.  
\]
\end{lem}

\begin{proof} 
(The following argument is 
a slight modification of the one given in 
\cite[Proof of Proposition 3.4]{KMT-arxiv-2}, 
where the case $r=2$ was treated.)  
By duality and 
by an appropriate extension of functions on 
$\Z^n$ to functions on $\R^n$, 
it is sufficient to prove the inequality 
\begin{align*}
\int_{(\R^n)^N} 
A_0(x_1+\cdots+x_N) 
\prod_{j=1}^N \langle x_{j}\rangle ^{a_j} A_{j}(x_{j})
\, dx_1 \cdots dx_N
\lesssim 
\|A_0\|_{ L^{r}(\R^n) }
\prod_{j=1}^N 
\|A_{j}\|_{ L^{2} (\R^n)}
\end{align*}
for 
all nonnegative functions 
$A_0, A_1, \dots, A_N$ on 
$\R^n$. 
In \cite[Proof of Proposition 3.4]{KMT-arxiv-2}, 
it is proved that 
the following inequality 
holds for 
all nonnegative functions 
$A_0, A_1, \dots, A_N, f_1, \dots, f_N$ on 
$\R^n$: 
\begin{align*}
&\int_{(\R^n)^N} 
A_0(x_1+\cdots+x_N) 
\prod_{j=1}^N f_{j}(x_{j}) A_{j}(x_{j})
\, dx_1 \cdots dx_N
\\&\lesssim 
\|A_0\|_{L^{q_0,s_0}(\R^n)}
\prod_{j=1}^N \|f_{j}\|_{L^{p_{j},s_{j}}(\R^n)} 
\|A_{j}\|_{L^{q_{j},t_{j}}(\R^n)}
\end{align*}
for 
\begin{align*}
&
\frac{1}{q_0} + \sum_{j=1}^N 
\left(\frac{1}{p_{j}}+\frac{1}{q_{j}}\right) 
=N, 
\\
&
0 < 1/q_0,1/p_{j}, 1/q_{j} < 1, \\
&
0 < {1}/{p_{j}}+{1}/{q_{j}} < 1,
\quad j=1,\dots,N,
\\&
s_0, s_j, t_j\in [1, \infty], \quad 
j=1,\dots, N,
\quad \text{and}
\quad 
\frac{1}{s_0} +
\sum_{j=1}^N 
\left(\frac{1}{s_{j}}+\frac{1}{t_{j}}\right) 
=1.
\end{align*}
We take 
\begin{align*}
&
q_0 = r, \quad 
q_1 = \cdots = q_{N} = 2, \\
&
s_0 = s_1 = \cdots = s_N = t_1= \cdots = t_{N-2} = \infty, \quad
t_{N-1} = t_N = 2.  
\end{align*}
Then the above result combined with the embedding 
$L^r \hookrightarrow L^{r,\infty}$, $r<\infty$, 
implies 
\begin{align*}
\int_{(\R^n)^N} 
A_0(x_1+\cdots+x_N) 
\prod_{j=1}^N f_{j}(x_{j}) A_{j}(x_{j})
\, dx_1 \cdots dx_N
\lesssim 
\|A_0\|_{ L^{r} }
\prod_{j=1}^N 
\|f_{j}\|_{ L^{p_{j},\infty} } 
\|A_{j}\|_{ L^{2} }
\end{align*}
for
\begin{align*}
\sum_{j=1}^N \frac{1}{p_{j}}
=\frac{\,N\,}{2} - \frac{1}{\,r\,} , \quad
0< \frac{1}{\,r\,}<1, 
\quad 
0 < \frac{1}{\,p_{j}\,} < \frac{1}{\,2\,} 
\quad (\, j=1,\dots,N\, ). 
\end{align*}
The desired inequality now follows if we take  
$p_j=-n/a_j$ and  
$f_j (x_j)=\langle x_j \rangle ^{a_j}$. 
\end{proof}

%%%============================
\begin{lem}\label{productLweak1}
Let $N \ge 2$. 
Then 
the following 
inequality 
holds for all nonnegative functions 
$A_1, \, \dots, \, A_N$ on $\Z^n$: 
\begin{equation}\label{General-Hilbert}
\sum_{
\nu_1, \,\dots, \,\nu_N \in \Z^n}
(1+ |\nu_1| + \cdots + |\nu_N| )^{-Nn/2} 
\prod_{j=1}^N 
A_j(\nu_j)
\lesssim \prod_{j=1}^N
\|A_j\|_{\ell^2 (\Z^n) } . 
\end{equation}
\end{lem}
%%====================================

\begin{proof}
If $N=2$ and $n=1$, 
then the inequality 
reads as 
\[
\sum_{
\nu_1, \,\nu_2 \in \Z}
(1+ |\nu_1| + |\nu_2| )^{-1} 
A_1(\nu_1) 
A_2(\nu_2) 
\lesssim 
\|A_1\|_{\ell^2 (\Z) }
\|A_1\|_{\ell^2 (\Z) }. 
\]
This inequality is known by the name of Hilbert's inequality and is 
derived from 
the $L^2$-boundedness of 
the Hilbert transform; 
see, {\it e.g.}, \cite[Chapter 11, Theorem 11.4.1]{Garling}. 
If $N=2$ and $n\ge 2$, then, 
since 
\[
(1+ |\nu_1| + |\nu_2| )^{-n} 
\le 
\prod_{k=1}^{n} 
(1+ |\nu_{1,k}| + |\nu_{2,k}| )^{-1}, 
\quad 
\nu_j = (\nu_{j,1}, \dots, \nu_{j,n}), 
\] 
the inequality can be proved 
by repeated application of Hilbert's inequality.

We shall prove the inequality for $N \ge 3$. 
Applying the Cauchy--Schwarz inequality 
to the sum $\sum_{\nu_1 \in \Z^n}$, 
we obtain 
\begin{align*}
&(\text{the left hand side of \eqref{General-Hilbert}})
\\
&\le 
\sum_{\nu_2, \dots, \nu_{N} \in \Z^n}
\| A_{1} \|_{\ell^2}
\bigg( \sum_{\nu_{1}\in \Z^n }
\big( 1+ |\nu_1| + \cdots + |\nu_N|
\big)^{-Nn} 
\bigg)^{1/2}
\, 
\prod_{j=2}^{N} A_{j} (\nu_{j}) 
\\&\approx
\sum_{\nu_2, \dots, \nu_{N} \in \Z^n} 
\| A_{1} \|_{\ell^2} 
{ \big( 1+ |\nu_2| + \cdots + |\nu_N|\big )^{-(N-1)n/2} }
\, 
\prod_{j=2}^{N} A_{j} (\nu_{j}).
\end{align*}
Repeating this process $N-2$ times, 
we obtain 
\begin{align*}
&(\text{the left hand side of \eqref{General-Hilbert}})
\\
&\lesssim
\prod_{j=1}^{N-2}
\| A_{j} \|_{\ell^2} 
\sum_{\nu_{N-1}, \nu_{N} \in \Z^n}
{ \left( 1+ |\nu_{N-1}|+|\nu_{N}| \right)^{-n} } 
\, 
A_{N-1} (\nu_{N-1}) A_{N} (\nu_{N})  
=: (\ast). 
\end{align*}
Finally using the case $N=2$ of 
\eqref{General-Hilbert}, 
we obtain 
$(\ast)\lesssim 
\prod_{j=1}^{N}
\| A_{j} \|_{\ell^2}$. 
\end{proof}

%%%%%%%%%%%%%%%%%%%%%%%%%%%%%%%%%%%%%%%%%%%%
\section{Proof of Theorem \ref{main-thm-sufficiency} }
\label{section3}
%%%%%%%%%%%%%%%%%%%%%%%%%%%%%%%%%%%%%%%%%%%

We use the following notation:   
$\vecxi=(\xi_1,\dots,\xi_N) \in (\R^n)^N$, 
$\vecnu=(\nu_1,\dots,\nu_N) \in (\Z^n)^N$, 
and 
$\veck=(k_1,\dots,k_N) \in (\Z^n)^N$. 
We also write 
$d\vecxi = d\xi_1 \cdots d\xi_N$ and 
$
\vecxi \cdot \veck = \xi_1 \cdot k_1 + \cdots + \xi_N \cdot k_N$.

\begin{proof}[Proof of Theorem \ref{main-thm-sufficiency}]

We first give some argument 
which can be applied 
to general multilinear pseudo-differential operators 
and prove some result concerning the estimate 
of the operators in Wiener amalgam spaces.

For the moment, we assume 
$\sigma \in S^{\vecm} (\R^n, N)$ 
with 
$\vecm \in \R^N$ or 
$\sigma \in S^{m} (\R^n, N)$ 
with 
$m \in \R$.

We shall follow the method of utilizing Fourier series expansion, 
which may go back at least to 
Coifman and Meyer \cite{CM-Ast, CM-AIF}. 
We take a function $\varphi$ such that 
\begin{equation*}
\varphi \in C_{0}^{\infty} (\R^n), \quad
\supp \varphi \subset [-1, 1]^{n}, 
\quad
\sum_{\nu\in\Z^n} \varphi (\xi-\nu) = 1  
\quad
(\xi \in \R^n)  
\end{equation*}
and 
decompose the symbol $\sigma$ as 
\begin{align*}
&
\sigma(x,\vecxi)=
\sum_{\vecnu \in \ZnN} 
\sigma_{\vecnu}(x,\vecxi), 
\\
&
\sigma_{\vecnu}(x,\vecxi)
=
\sigma(x,\vecxi) 
\prod_{j=1}^{N} \varphi (\xi_j-\nu_j). 
\end{align*}
If we define 
\[
S_{\boldsymbol{\nu}} (x,\vecxi) = 
\sum_{\vecmu \in \Z^n} \sigma_{\vecnu}(x,\vecxi - 2\pi \vecmu),  
\]
then 
$S_{\boldsymbol{\nu}}(x, \vecxi)$ 
is a $2\pi \Z^{Nn}$-periodic function 
with respect to the variable 
$\vecxi$ 
and $S_{\boldsymbol{\nu}}(x, \vecxi)=
\sigma_{\boldsymbol{\nu}}(x, \vecxi)$ for 
$\vecxi \in \boldsymbol{\nu} + [-\pi, \pi]^{Nn}$. 
Hence 
we can take a function $\widetilde{\varphi}$ such that 
\begin{equation*}
\widetilde\varphi \in C_{0}^{\infty} (\R^n), \quad
0 \le \widetilde\varphi \le 1, \quad 
\supp \widetilde\varphi \subset [-\pi, \pi]^n, 
\quad 
\widetilde\varphi = 1 \;\;\textrm{on}\;\; [-1, 1]^n, 
\end{equation*}
and 
\begin{equation*}
\sigma_{\boldsymbol{\nu}} (x,\vecxi)
=
S_{\boldsymbol{\nu}} (x,\vecxi)
\prod_{j=1}^{N} \widetilde\varphi(\xi_j-\nu_j). 
\end{equation*}
Expanding $S_{\boldsymbol{\nu}}(x, \vecxi)$ into 
the Fourier series with respect to $\vecxi$, 
we obtain   
\begin{align*}
&
\sigma_{\boldsymbol{\nu}}  (x,\vecxi)
= \sum_{\boldsymbol{k} \in\ZnN}
e^{i \vecxi \cdot \veck}
P_{\boldsymbol{\nu},\boldsymbol{k}} (x)  
\prod_{j=1}^{N} \widetilde\varphi(\xi_j-\nu_j), 
\\
&
P_{\boldsymbol{\nu},\boldsymbol{k}} (x) 
= 
\frac{1}{(2 \pi )^{Nn}} \int_{\boldsymbol{\nu}+[-\pi , \pi ]^{Nn}}
e^{-i \veck \cdot \veceta}
\sigma_{\boldsymbol{\nu}} (x,\veceta) 
\,d\veceta 
\end{align*}
Integration by parts gives 
\begin{align*}
&P_{\boldsymbol{\nu},\boldsymbol{k}} (x) 
=
\langle \boldsymbol{k} \rangle^{-2L}
Q_{\boldsymbol{\nu},\boldsymbol{k}} (x), 
\\
&Q_{\boldsymbol{\nu},\boldsymbol{k}} (x)
 =
\frac{1}{(2 \pi )^{Nn}} 
\int_{\boldsymbol{\nu}+[-\pi , \pi ]^{Nn}}
e^{-i \veck \cdot \veceta}
\big( I-\Delta_{\veceta} \big)^{L} 
\big\{ \sigma_{\boldsymbol{\nu}} (x,\veceta) \big\} 
\,d\veceta 
\end{align*}
We further decompose 
$Q_{\boldsymbol{\nu},\boldsymbol{k}}$ 
by using the partition of unity of  
Lemma \ref{PSdec} 
as 
\begin{align*}
&
Q_{\boldsymbol{\nu},\boldsymbol{k}} (x) = 
\sum_{\ell \in \Z^n} 
\langle \ell \rangle^{-2M} 
Q_{\boldsymbol{\nu},\boldsymbol{k},\ell} (x), 
\\
&
Q_{\boldsymbol{\nu},\boldsymbol{k},\ell} (x)= 
\calF^{-1} 
\left[ 
\chi_\ell (\zeta) 
\langle \zeta \rangle^{2M} 
 \widehat{Q_{\boldsymbol{\nu},\boldsymbol{k}}}(\zeta) 
 \right](x)
=
(\calF^{-1}\chi_{\ell} (x))\ast (I- \Delta_{x})^{M} 
Q_{\boldsymbol{\nu},\boldsymbol{k}} (x).  
\end{align*}
Thus 
$\sigma_{\nu}$ is written as 
\begin{equation}\label{sigmanu}
\sigma_{\boldsymbol{\nu}} (x,\vecxi)
= 
\sum_{\boldsymbol{k} \in\ZnN}\, 
\sum_{\ell \in \Z^n}\, 
\langle \boldsymbol{k} \rangle^{-2L}
\langle \ell \rangle^{-2M}
Q_{\boldsymbol{\nu},\boldsymbol{k}, \ell} (x)
e^{i \vecxi \cdot \veck} 
\prod_{j=1}^{N} \widetilde\varphi(\xi_j-\nu_j). 
\end{equation}

By \eqref{sigmanu}, 
\begin{align*}
&
T_{\sigma_{\boldsymbol{\nu}}}(f_1,\dots,f_N)(x)
\\
&
= 
\sum_{\boldsymbol{k} \in\ZnN}\, 
\sum_{\ell \in \Z^n}\, 
\langle \boldsymbol{k} \rangle^{-2L} 
\langle \ell \rangle^{-2M} 
Q_{\boldsymbol{\nu},\boldsymbol{k},\ell} (x)
\prod_{j=1}^{N}
\widetilde\varphi(D-\nu_j)f_j(x+ k_j) . 
\end{align*}
By writing  
\[
F^{j}_{\nu_j,k_j}(x) = \widetilde\varphi(D-\nu_j)f_j(x+ k_j), 
\]
we have 
\begin{equation*}
T_{\sigma_{\boldsymbol{\nu}}}(f_1,\dots,f_N)(x)
= \sum_{\boldsymbol{k} \in\ZnN}\, 
\sum_{\ell \in \Z^n}\, 
\langle \boldsymbol{k} \rangle^{-2L} 
\langle \ell \rangle^{-2M} 
Q_{\boldsymbol{\nu},\boldsymbol{k},\ell} (x)
\prod_{j=1}^{N}
F^{j}_{\nu_j,k_j}(x). 
\end{equation*}
Thus, since 
$\sigma = \sum_{\boldsymbol{\nu}} \sigma_{\boldsymbol{\nu}}$,  
\begin{equation}\label{TsigmaF}
T_{\sigma}(f_1,\dots,f_N)(x)
= \sum_{\boldsymbol{k} \in\ZnN}\, 
\sum_{\ell \in \Z^n}\, 
\langle \boldsymbol{k} \rangle^{-2L} 
\langle \ell \rangle^{-2M} 
\sum_{\boldsymbol{\nu}\in \ZnN} 
Q_{\boldsymbol{\nu},\boldsymbol{k},\ell} (x)
\prod_{j=1}^{N}
F^{j}_{\nu_j,k_j}(x). 
\end{equation}

The functions $Q_{\boldsymbol{\nu},\boldsymbol{k}, \ell} (x)$ have 
the following estimate. 
If $\sigma \in S^{\vecm}_{0,0}(\R^n, N)$ 
with $\vecm \in \R^N$, 
then the symbol 
$\sigma_{\vecnu}(x,\vecxi)$ satisfies 
\begin{equation*}
\left|\partial^{\alpha}_x 
\partial^{\beta_1}_{\xi_1} \cdots \partial^{\beta_N}_{\xi_N} 
\sigma_{\vecnu}(x,\vecxi)
\right|
\lesssim
\prod_{j=1}^{N}
\big( 1 + |\nu_j| \big)^{m_j}   
\end{equation*}
and hence $Q_{\boldsymbol{\nu},\boldsymbol{k}} (x)$ satisfies 
\[
\left| 
(I- \Delta_{x})^{M} 
Q_{\boldsymbol{\nu},\boldsymbol{k}} (x)
\right|
\lesssim 
\prod_{j=1}^{N}
\big( 1 + |\nu_j| \big)^{m_j},  
\]
which combined with 
$\|\calF^{-1}\chi_{\ell}\|_{L^1}\lesssim 1$ (Lemma \ref{PSdec}) 
implies 
\begin{equation}\label{Qnukellvecm}
\left| 
Q_{\boldsymbol{\nu},\boldsymbol{k}, \ell} (x) \right| 
\lesssim 
\prod_{j=1}^{N}
\big( 1 + |\nu_j| \big)^{m_j}. 
\end{equation}
If $\sigma \in S_{0,0}^{m}(\R^n,N)$ 
with $m\in R$, 
then 
\begin{equation*}
\left|\partial^{\alpha}_x 
\partial^{\beta_1}_{\xi_1} \cdots \partial^{\beta_N}_{\xi_N} 
\sigma_{\vecnu}(x,\vecxi)\right|
\lesssim
\big( 1 + |\nu_1|+ \cdots +|\nu_N| \big)^{m}  
\end{equation*}
and hence, by 
the same reason as above,  
\begin{equation}\label{Qnukellscalarm}
\left| 
Q_{\boldsymbol{\nu},\boldsymbol{k}, \ell} (x) \right| 
\lesssim 
\big( 1 + |\nu_1|+ \cdots + |\nu_N| \big)^{m}. 
\end{equation}
Notice that 
the implicit constants in 
\eqref{Qnukellvecm} 
and 
\eqref{Qnukellscalarm} 
do not depend on 
$\veck$ and $\ell$.

Let us consider 
the estimate of 
$W^{s,t}$-quasinorm of 
$T_{\sigma}(f_1,\dots,f_N)$ for $s,t \in (0, \infty]$. 
If $s,t \in (0, \infty]$ are given, 
taking the numbers 
$L$ and $M$ sufficiently large and using 
\eqref{TsigmaF}, we obtain 
\begin{align*}
\| T_{\sigma}(f_1,\dots,f_N) \|_{W^{s,t}} 
\lesssim
\sup_{\boldsymbol{k},\ell}
\left\| \sum_{\boldsymbol{\nu}} 
Q_{\boldsymbol{\nu},\boldsymbol{k},\ell}  
\prod_{j=1}^{N} 
F^{j}_{\nu_j,k_j}
\right\|_{W^{s,t}}. 
\end{align*}
Notice that 
\begin{align}
&\supp \calF Q_{\boldsymbol{\nu},\boldsymbol{k},\ell}
\subset \{\zeta\in\R^n: |\zeta - \ell | \lesssim 1 \} , 
\label{supp-hatQ}
\\
&
\supp \calF F^{j}_{\nu_j,k_j}
\subset \{\zeta\in\R^n: |\zeta - \nu_j | \lesssim 1 \} 
\label{supp-hatF}
\end{align}
and recall that the function $\kappa$ used in the definition of 
$W^{s,t}$-quasinorm has compact support. 
Hence 
we have 
\begin{align*}
&
\left\| 
\sum_{\boldsymbol{\nu}} 
Q_{\boldsymbol{\nu},\boldsymbol{k},\ell}  
\prod_{j=1}^{N} 
F^{j}_{\nu_j,k_j}
\right\|_{W^{s,t}}
\\
&
=
\left\| \left\|
\sum_{\vecnu:|\nu_1+\cdots+\nu_N+\ell-\mu|\lesssim 1} 
\kappa(D-\mu)\left[
Q_{\boldsymbol{\nu},\boldsymbol{k},\ell}
\prod_{j=1}^{N} F^{j}_{\nu_j,k_j}\right]
\right\|_{\ell^{t}_{\mu}(\Z^n)} \right\|_{L^{s}(\R^n)}
\\
&
\lesssim 
\sum_{\tau:|\tau|\lesssim 1} 
\left\| \left\| 
\kappa(D-\mu+\tau-\ell)\left[
\sum_{\vecnu:\nu_1+\cdots+\nu_N=\mu}
Q_{\boldsymbol{\nu},\boldsymbol{k},\ell}
\prod_{j=1}^{N} F^{j}_{\nu_j,k_j}\right]
\right\|_{\ell^{t}_{\mu}(\Z^n)} \right\|_{L^{s}(\R^n)}
\\
&
=:(\ast). 
\end{align*}
We write 
\[
h_{\mu}=
\sum_{\vecnu:\nu_1+\cdots+\nu_N=\mu}
Q_{\boldsymbol{\nu},\boldsymbol{k},\ell}
\prod_{j=1}^{N} F^{j}_{\nu_j,k_j}.  
\]
Then 
$\supp \widehat{h_\mu}
\subset 
\{ \zeta \mid |\zeta - \mu -  \ell| \lesssim 1\}$ 
by 
\eqref{supp-hatQ} and \eqref{supp-hatF}.   
Recall that 
$\supp \kappa$ is also compact. 
Hence by 
Nikol'skij's inequality (see, {\it e.g.}, 
\cite[Section 1.3.2, Remark 1]{triebel 1983}),  
we have 
\begin{align}
\left| \kappa(D-\mu+\tau-\ell)h_{\mu}(x) \right|
&\le
\left\| (\calF^{-1}\kappa) (y) h_\mu(x-y) \right\|_{L^1_y}
\nonumber 
\\
&
\lesssim 
\left\| (\calF^{-1}\kappa) (y) h_\mu(x-y) \right\|_{L^{\min(1,s,t)}_y}. 
\label{forminko}
\end{align}
The implicit constant in \eqref{forminko} 
can be taken depending only on $s, t$, $n$, 
and the diameters 
of $\supp \widehat{h_\mu}$ and 
$\supp \kappa$;  
in particular, the inequality \eqref{forminko} 
holds uniformly with respect to 
$\mu,\tau$, and $\ell$. 
By \eqref{forminko} and Minkowski's inequality, we have 
\[
(\ast)
\lesssim 
\left\| \left\| 
\left\| (\calF^{-1}\kappa) (y) h_\mu(x-y) \right\|_{L^{\min(1,s,t)}_y}
\right\|_{\ell^{t}_{\mu}(\Z^n)} \right\|_{L^{s}_{x}(\R^n)}
\lesssim
\left\| \left\| 
h_\mu
\right\|_{\ell^{t}_{\mu}(\Z^n)} \right\|_{L^{s}(\R^n)}.
\]
To sum up, we see that the inequality 
\begin{equation}\label{beforebdd} 
\| T_{\sigma}(f_1,\dots,f_N) \|_{W^{s,t}}
\lesssim 
\sup_{\boldsymbol{k},\ell}
\left\| \left\| 
\sum_{\vecnu:\nu_1+\cdots+\nu_N=\mu}
Q_{\boldsymbol{\nu},\boldsymbol{k},\ell}
\prod_{j=1}^{N} F^{j}_{\nu_j,k_j}
\right\|_{\ell^{t}_{\mu}(\Z^n)} \right\|_{L^{s}(\R^n)}
\end{equation}
holds for each $s,t \in (0,\infty]$. 
\vs

Now we shall proceed to the estimate  
of the operators considered in Theorem 
\ref{main-thm-sufficiency}. 
We introduce some notation.  
If 
$p_1, \dots, p_N\in (0, \infty]$ 
and function spaces 
$X_1, \dots, X_N, Y$ on $\R^n$ 
are given, then 
we define $p_0$ by 
$1/p_0=1/p_1 + \cdots + 1/p_N$, 
define  the sets $J$ and $J^c$ by 
\[
J=
\left\{ 
j \in \{1, \dots, N\} 
\mid 2 \le p_j \le \infty
\right\}, 
\quad 
J^{c}=\left\{ 
j \in \{1, \dots, N\} 
\mid 0<p_j <2
\right\}, 
\] 
and write 
the boundedness $T_{\sigma}: X_1 \times \cdots \times X_N \to Y$ 
as 
\[
T_{\sigma}: \prod_{j\in J} X_j \times \prod_{j\in J^c} X_j 
\to Y.   
\]

We devide the rest of the arguments into three parts. 
\vs

%%=======================
{\it Part I: Proof of Theorem \ref{main-thm-sufficiency} (1) 
for the case $0<p\le 2$.}  
%%========================

Assume 
$\vecm \in \R^N$ satisfies 
the conditions of 
Theorem \ref{main-thm-sufficiency} (1) 
with $0<p\le 2$  
and assume  
$\sigma \in S_{0,0}^{\vecm}(\R^n,N)$.  
We shall prove the estimate 
\begin{equation}\label{ple2WA}
T_{\sigma} : 
\prod_{j \in J} W^{p_j,2} 
\times 
\prod_{j \in J^c } W^{p_j,2}_{n(1/2-1/p_j)}  
\to 
W^{p_0,2}. 
\end{equation}
If this is proved, 
then combining it with 
the embeddings 
\begin{align}
&
h^{p_j} \hookrightarrow W^{p_j,2} 
\quad \text{for}\quad  j\in J, 
\label{embeddingJ}
\\
&
h^{p_j} \hookrightarrow W^{p_j,2}_{n(1/2 - 1/p_j)} 
\quad \text{for}\quad  j\in J^c, 
\label{embeddingJc}
\\
&
W^{p_0,2}\hookrightarrow W^{p,2} \hookrightarrow h^p
\quad \text{for}\quad p_0 \le p\le 2
\nonumber
\end{align}
(see Lemma \ref{Waembd}), 
we obtain the desired boundedness 
$T_{\sigma}: 
h^{p_1}\times \cdots \times h^{p_N} 
\to h^p$. 
Since the embedding 
$bmo \hookrightarrow W^{\infty, 2}$ also holds, 
if $p_j=\infty$ 
then we can replace 
$h^{p_j}$ by $bmo$.

To prove 
\eqref{ple2WA}, 
we use 
\eqref{beforebdd} with $s=p_0$, $t=2$, 
and use \eqref{Qnukellvecm}. 
Then we obtain 
\begin{align*}
&
\left\| 
T_{\sigma}(f_1, \dots, f_N)
\right\|_{W^{p_0, 2}}
\\
&
\lesssim 
\sup_{\boldsymbol{k}} 
\left\| \left\| 
\, 
\sum_{\vecnu:\nu_1+\cdots+\nu_N=\mu}
\, 
\prod_{j=1}^{N}
\big( 1 + |\nu_j| \big)^{m_j} 
\left| F^{j}_{\nu_j,k_j} \right|
\right\|_{\ell^{2}_{\mu}} \right\|_{L^{p_0}}
\\&=
\sup_{\boldsymbol{k}} 
\left\| \left\| 
\, 
\sum_{\vecnu:\nu_1+\cdots+\nu_N=\mu}
\, 
\prod_{j \in J} \langle \nu_j \rangle^{m_j}
\, 
\prod_{j \in J^c} \langle \nu_j \rangle^{m_j-n(1/2-1/p_j)} 
\, 
\prod_{j \in J} 
\left| F^{j}_{\nu_j,k_j} \right|
\, 
\prod_{j \in J^c} 
\left| \widetilde{F}^{j}_{\nu_j,k_j} \right|
\right\|_{\ell^{2}_{\mu}} \right\|_{L^{p_0}}
\\
&
=: (\dag), 
\end{align*}
where 
$\widetilde{F}^{j}_{\nu_j,k_j} = 
\langle \nu_j \rangle^{n(1/2-1/p_j)} F^{j}_{\nu_j,k_j}$. 
Now using 
Lemma \ref{productLweakp'} with $r=2$ 
and using H\"older's inequality with 
the indices 
$1/p_0 = 1/p_1+\cdots+1/p_N$, 
we obtain 
\begin{align*}
(\dag)
&\lesssim
\sup_{\boldsymbol{k}}\, 
\left\| 
\prod_{j \in J} \left\| F^{j}_{\nu_j,k_j} \right\|_{\ell^2_{\nu_j}}
\prod_{j \in J^c} \left\| \widetilde{F}^{j}_{\nu_j,k_j} \right\|_{\ell^2_{\nu_j}}
\right\|_{L^{p_0}}
\\
&
\lesssim 
\sup_{\boldsymbol{k}}\, 
\prod_{j \in J} 
\left\| \left\| 
F^{j}_{\nu_j,k_j} 
\right\|_{\ell^2_{\nu_j}} \right\|_{L^{p_j}}
\prod_{j \in J^c} 
\left\| \left\| 
\widetilde{F}^{j}_{\nu_j,k_j} 
\right\|_{\ell^2_{\nu_j}} \right\|_{L^{p_j}}. 
\end{align*}
Finally, the choice of $\widetilde{\varphi}$ implies  
\begin{align*}
\left\| \left\| 
F^{j}_{\nu_j,k_j} 
\right\|_{\ell^2_{\nu_j}} \right\|_{L^{p_j}}
&=
\left\| \left\| 
\widetilde\varphi(D-\nu_j)f_j(x+k_j) 
\right\|_{\ell^2_{\nu_j}} \right\|_{L^{p_j}}
\approx \| f_j \|_{W^{p_j,2}}, 
\\
\left\| \left\| 
\widetilde{F}^{j}_{\nu_j,k_j} 
\right\|_{\ell^2_{\nu_j}} \right\|_{L^{p_j}}
&=
\left\| \left\| \langle \nu_j \rangle^{n(1/2-1/p_j)} 
\widetilde\varphi(D-\nu_j)f_j(x+k_j) 
\right\|_{\ell^2_{\nu_j}} \right\|_{L^{p_j}}
\approx \| f_j \|_{ W^{p_j,2}_{n(1/2-1/p_j)} }, 
\end{align*}
and thus the estimate \eqref{ple2WA} follows. 
\vs

%%========================
{\it Part II: Proof of Theorem \ref{main-thm-sufficiency} (1) 
for the case $2<p<\infty$.} 
%%========================

Assume 
$\vecm \in \R^N$ satisfies 
the conditions of 
Theorem \ref{main-thm-sufficiency} (1) 
with $2<p< \infty$  
and assume  
$\sigma \in S_{0,0}^{\vecm}(\R^n,N)$.  
In this case, we shall prove the estimate 
\begin{equation}\label{pge2WA}
T_{\sigma} : 
\prod_{j \in J} W^{p_j,2} 
\times 
\prod_{j \in J^c } W^{p_j,2}_{n(1/2-1/p_j)}  
\to 
W^{p_0,p'}. 
\end{equation}
Combining this with 
the embeddings 
\eqref{embeddingJ}, \eqref{embeddingJc}, 
and 
\[
W^{p_0,p'} \hookrightarrow 
W^{p,p'}
\hookrightarrow L^{p}=h^{p}, 
\quad 
p_0\le p, \;\;  2<p<\infty 
\]
(see Lemma \ref{Waembd}), 
we obtain 
$T_{\sigma}: 
h^{p_1}\times \cdots \times h^{p_N} 
\to h^p$. 
If $p_j=\infty$, 
then we can replace 
$h^{p_j}$ by $bmo$.

To prove \eqref{pge2WA}, 
we use 
\eqref{beforebdd} with $s=p_0$, $t=p^{\prime}$, 
and also use 
\eqref{Qnukellvecm}. 
This yields 
\begin{align*}
&
\left\| 
T_{\sigma}(f_1, \dots, f_N)
\right\|_{W^{p_0, p^{\prime}}}
\\
&
\lesssim 
\sup_{\boldsymbol{k}} 
\left\| \left\| 
\, 
\sum_{\vecnu:\nu_1+\cdots+\nu_N=\mu}
\, 
\prod_{j \in J} \langle \nu_j \rangle^{m_j}
\, 
\prod_{j \in J^c} \langle \nu_j \rangle^{m_j-n(1/2-1/p_j)} 
\, 
\prod_{j \in J} 
\left| F^{j}_{\nu_j,k_j} \right|
\, 
\prod_{j \in J^c} 
\left| \widetilde{F}^{j}_{\nu_j,k_j} \right|
\right\|_{\ell^{p^{\prime}}_{\mu}} \right\|_{L^{p_0}}
\\
&=: (\dag\dag), 
\end{align*}
where 
$\widetilde{F}^{j}_{\nu_j,k_j}$ 
is the same as in Part I. 
Using 
Lemma \ref{productLweakp'} with $r=p$ 
and using H\"older's inequality with 
the indices 
$1/p_0 = 1/p_1+\cdots+1/p_N$, 
we obtain 
\begin{align*}
(\dag\dag)
&\lesssim
\sup_{\boldsymbol{k}}\, 
\left\| 
\prod_{j \in J} \left\| F^{j}_{\nu_j,k_j} \right\|_{\ell^2_{\nu_j}}
\prod_{j \in J^c} \left\| \widetilde{F}^{j}_{\nu_j,k_j} \right\|_{\ell^2_{\nu_j}}
\right\|_{L^{p_0}}
\\
&
\lesssim 
\sup_{\boldsymbol{k}}\, 
\prod_{j \in J} 
\left\| \left\| 
F^{j}_{\nu_j,k_j} 
\right\|_{\ell^2_{\nu_j}} \right\|_{L^{p_j}}
\prod_{j \in J^c} 
\left\| \left\| 
\widetilde{F}^{j}_{\nu_j,k_j} 
\right\|_{\ell^2_{\nu_j}} \right\|_{L^{p_j}}
\\
&
\approx 
\prod_{j \in J} 
\| f_j \|_{W^{p_j,2}}
\prod_{j \in J^c} 
\| f_j \|_{ W^{p_j,2}_{n(1/2-1/p_j)} }. 
\end{align*}
Thus 
\eqref{pge2WA} is proved. 
\vs

%%========================
{\it Part III: Proof of Theorem \ref{main-thm-sufficiency} (2).} 
%%========================

Assume $p=\infty$, 
$m \in \R$ is given as in 
Theorem \ref{main-thm-sufficiency} (2), 
and assume  
$\sigma \in S_{0,0}^{m}(\R^n,N)$.  
We shall prove the estimate 
\begin{equation}\label{pinftyWA}
T_{\sigma} : 
\prod_{j \in J} W^{p_j,2} 
\times 
\prod_{j \in J^c } W^{p_j,2}_{n(1/2-1/p_j)}  
\to 
W^{p_0,1}.    
\end{equation}
Combining this estimate 
with the embeddings 
\eqref{embeddingJ}, \eqref{embeddingJc}, 
and 
\[
W^{p_0,1} \hookrightarrow 
W^{\infty,1}
\hookrightarrow L^{\infty}
\]
(see Lemma \ref{Waembd}), 
we obtain the desired boundedness 
$T_{\sigma}: 
h^{p_1}\times \cdots \times h^{p_N} 
\to L^{\infty}$. 
If $p_j=\infty$, 
then we can replace 
$h^{p_j}$ by $bmo$.

By \eqref{Qnukellscalarm}, 
we have 
\begin{equation*}
\left| 
Q_{\boldsymbol{\nu},\boldsymbol{k}, \ell} (x) \right| 
\lesssim 
\big( 1 + |\nu_1|+ \cdots + |\nu_N| \big)^{m}
\end{equation*}
Observe that  
$m = -Nn/2 + \sum_{j\in J^c} (n/2- n/p_j)$ in the present case. 
Hence the above inequality implies 
\begin{equation}\label{Qnukellscalarm-2}
\left| 
Q_{\boldsymbol{\nu},\boldsymbol{k}, \ell} (x) \right| 
\lesssim 
\big( 1 + |\nu_1|+ \cdots + |\nu_N| \big)^{-Nn/2}
\prod_{j\in J^c} \langle \nu_j \rangle^{n (1/2 - 1/p_j)}. 
\end{equation}

We use 
\eqref{beforebdd} with $s=p_0$, $t=1$,   
and use \eqref{Qnukellscalarm-2} to obtain 
\begin{align*}
&
\left\| 
T_{\sigma}(f_1, \dots, f_N)
\right\|_{W^{p_0, 1}}
\\
&
\lesssim 
\sup_{\boldsymbol{k}} 
\left\| 
\left\| 
\, 
\sum_{\vecnu:\nu_1+\cdots+\nu_N=\mu}
\, 
(1+ |\nu_1|+ \cdots + |\nu_N|)^{-Nn/2}
\, 
\prod_{j \in J} \left| F^{j}_{\nu_j,k_j} \right|
\, 
\prod_{j \in J^c} \left| \widetilde{F}^{j}_{\nu_j,k_j} \right|
\, 
\right\|_{\ell^{1}_{\mu}} 
\right\|_{L^{p_0}}
\\
&=: (\dag\dag\dag), 
\end{align*}
where 
$\widetilde{F}^{j}_{\nu_j,k_j}$ is the same as in Part I. 
Now by Lemma \ref{productLweak1} and 
H\"older's inequality, 
we obtain 
\begin{align*}
(\dag\dag\dag)
&\lesssim 
\sup_{\veck} 
\left\| 
\, 
\prod_{j \in J} 
\left\| F^{j}_{\nu_j,k_j} \right\|_{\ell^2_{\nu_j}}
\prod_{j \in J^c} 
\left\| \widetilde{F}^{j}_{\nu_j,k_j} \right\|_{\ell^2_{\nu_j}}
\, 
\right\|_{L^{p_0}}
\\
&
\le 
\sup_{\veck} 
\prod_{j \in J} 
\left\| 
\left\| F^{j}_{\nu_j,k_j} \right\|_{\ell^2_{\nu_j}}
\right\|_{L^{p_j}}
\prod_{j \in J^c} 
\left\| \left\| \widetilde{F}^{j}_{\nu_j,k_j} \right\|_{\ell^2_{\nu_j}}
\right\|_{L^{p_j}}
\\
&
\approx 
\prod_{j \in J} 
\| f_j \|_{W^{p_j,2}}
\prod_{j \in J^c} 
\| f_j \|_{ W^{p_j,2}_{n(1/2-1/p_j)} }. 
\end{align*}
Thus 
\eqref{pinftyWA} is proved. 
This completes the proof of 
Theorem \ref{main-thm-sufficiency}. 
\end{proof}

%%%%%%%%%%%%%%%%%%%%%%%%%%%%%%%%%%%%%%%%%%%%
\section{Proof of Theorem 
\ref{main-thm-necessity}}
\label{section4}
%%%%%%%%%%%%%%%%%%%%%%%%%%%%%%%%%%%%%%%%%%%

The key fact we shall use here is that 
for each $p_j\in (1, \infty)$ 
there exists a function 
$f_j \in L^{p_j} (\R^n)$ whose 
Fourier transform satisfies  
\begin{equation*}
\left|\widehat{f_j}(\xi)
\right| 
\approx 
\begin{cases}
{|\ell|^{n/p_j -n - \epsilon}} & \text{(when $1<p_j < 2$)} \\
{|\ell|^{-n/2 - \epsilon}} & \text{(when $2\le p_j <\infty$)} 
\end{cases}
\end{equation*}
for $\ell \in \Z^n\setminus \{0\}$ and $|\xi - \ell|\ll 1$. 
If $1<p_j \le 2$, 
the function $f_j$ can be easily found by setting 
$f_j (x)\approx |x|^{-n/p_j + \epsilon}$ for $|x|$ small. 
If $2<p_j < \infty$, the existence of $f_j$ is a classical fact 
but is not so elementary.  
Here to give a precise argument, 
we use the following result due to Wainger \cite{Wainger}.

%%%==================================================================
\begin{lem}[{\cite[Theorem 10]{Wainger}}]
\label{Wainger}
If $1\le p \le \infty$, 
$0 < a < 1$, and 
$b>\frac{\,n\,}{2}+ (1-a) (\frac{n}{\,2\,}- \frac{n}{\,p\,})$, 
then there exists a function $g_{a,b} \in L^p ([-\pi, \pi]^n)$ 
whose Fourier coefficients are given by 
\begin{equation*}
\frac{1}{\,(2\pi)^{n}\, }\int_{[-\pi , \pi]^n} 
g_{a,b}(x) e^{-i \ell \cdot x}\,dx
=
\begin{cases}
{|\ell |^{-b}e^{2\pi i|\ell |^a}} & 
{\quad\text{for}\quad \ell \in \Z^n \setminus \{0\}}, 
\\ 
{0} & {\quad\text{for}\quad \ell =0}.
\end{cases}
\end{equation*}
\end{lem}
%%%==================================================================

Using this lemma, we shall prove the lemma below, 
which contains the 
essential part of Theorem \ref{main-thm-necessity}.

%%===========================================
\begin{lem}\label{lem1-necessity}
Suppose $N\ge 1$, 
$1< p_1, \dots, p_N<\infty$, 
$0<p<\infty$, and $m\in \R$ satisfy 
\begin{equation}\label{opxLp}
\op(S^{m}_{0,0}(\R^n, N)^{\times}) \subset
B(L^{p_1} \times \cdots \times L^{p_N} \to L^{p}). 
\end{equation}
Then \eqref{criticalm} holds. 
\end{lem}
%%===================================

\begin{proof} 
We set 
$J= \{j\in \{1, \dots, N\} \mid 2 \le p_j < \infty\}$ 
and 
$J^c= \{j\in \{1, \dots, N\} \mid 1< p_j < 2\}$. 
The proof will be divided into three cases. 
\vs

%%===============
{\it Case I: $0<p\le 2$.}  
%%===============

From 
\eqref{opxLp}, 
the closed graph theorem implies that 
there exists a positive integer $M$ and a 
constant $C$ 
such that 
\begin{equation} \label{closedgraphthm}
\|T_\sigma\|_{L^{p_1} \times \cdots \times L^{p_N} \to L^{p}}
\le
C
\max_{ |\beta_1|, \cdots, |\beta_N| \le M}
\left\|
(1+|\vecxi|)^{-m}
\partial_{\vecxi}^{\vecbeta}
\sigma(\vecxi)
\right\|_{L^\infty(\R^{Nn})}
\end{equation}
for all $\sigma \in S^{m}_{0,0}(\R^n, N)^{\times}$.

Let $\varphi, \widetilde{\varphi} \in \calS(\R^n)$ be functions 
such that 
\begin{align*}
&\supp \varphi \subset [-1/4, 1/4]^n,
\quad
|\calF^{-1}\varphi(x)| \ge 1 \ \text{on} \ [-\pi, \pi]^n,
\\
&\supp \widetilde{\varphi} \subset [-1/2, 1/2]^n,
\quad
\widetilde{\varphi} = 1 \ \text{on} \ [-1/4, 1/4]^n. 
\end{align*}
Let 
$\{c_{\veck}\}_{\veck \in (\Z^n)^N} $ be 
a sequence of complex numbers 
satisfying 
$\sup_{\veck} |c_{\veck}| \le 1$ 
and let  
$D$ be a finite subset of $(\Z^{n})^N$.  
We consider the multiplier 
\[ 
\sigma(\vecxi)
=
\sum_{\veck \in D}
c_{\veck} (1+ |\veck|)^{m} 
\prod_{j=1}^{N}
\varphi(\xi_j -k_j). 
\]
As far as $\sup_{\veck} |c_{\veck}| \le 1$, 
the multiplier $\sigma$ satisfies 
\begin{equation*}
\left|
\partial_{\vecxi}^{\vecbeta}
\sigma(\vecxi) 
\right|
\le c 
(1+|\vecxi|)^{m}
\end{equation*}
with a constant $c$ 
independent of 
$\{ c_{\veck}\} $ and $D$, 
and hence \eqref{closedgraphthm} implies 
\begin{equation} \label{UniformboundTsigma}
\|T_\sigma\|_{L^{p_1} \times \cdots \times L^{p_N} \to L^{p}}
\lesssim 1 
\end{equation}
with the implicit constant independent of 
$\{c_{\veck} \}$ and $D$.

Let $g_{a,b}$ be the function given in 
Lemma \ref{Wainger}. 
We extend $g_{a,b}$ to 
the $2\pi \Z^n$-periodic function 
on $\R^n$ and denote it by the same symbol $g_{a,b}$. 
Using this $g_{a,b}$, we define the 
functions $f_{a_j, b_j}$, 
$j=1, \dots, N$, on $\R^n$ as follows. 
We take 
$\epsilon > 0$ and $0< a_j < 1$, and define 
\begin{equation*}
b_j = \frac{\,n\,}{2} + (1-a_j ) \bigg( 
\frac{n}{\,2\,} -\frac{n}{\,p_j\,}\bigg) + \epsilon  
\end{equation*}
and 
\begin{equation*} 
f_{a_j, b_j}(x) = g_{a_j, b_j}(x) \calF^{-1}\widetilde{\varphi} (x). 
\end{equation*}
From 
Lemma \ref{Wainger}, 
it follows that  
\begin{equation}\label{fjest}
f_{a_j, b_j} \in L^{p_j} (\R^n)
\end{equation}
and its Fourier transform is given by 
\begin{equation*}
\widehat{f_{a_j, b_j}}(\xi_j)
=
\sum_{\ell_j \in \Z^n \setminus \{0\}} 
|\ell_j|^{-b_j} e^{2\pi i |\ell_j|^{a_j}} 
\widetilde{\varphi}(\xi_j - \ell_j). 
\end{equation*}
From our choice of 
$\varphi$ and $\widetilde{\varphi}$, 
we have 
\begin{align*}
T_\sigma(f_{a_1, b_1}, \cdots, f_{a_N, b_N})(x)
&=
\sum_{\veck \in D}
c_{\veck}
(1+ |\veck|)^{m}
\prod_{j=1}^{N}
|k_j|^{-b_j} 
e^{2\pi i|k_j|^{a_j}}
\calF^{-1} \big( \varphi (\cdot  - k_j ) \big) (x)
\\
&=
\Big\{ \calF^{-1} \varphi (x)\Big\}^{N}
\sum_{\veck \in D}
c_{\veck}
(1+ |\veck|)^{m}
\prod_{j=1}^{N}
|k_j|^{-b_j} 
e^{2\pi i|k_j|^{a_j}}
e^{i k_j \cdot x}. 
\end{align*}

Let 
$\{r_k(\omega)\}_{k \in \Z^n}$ be mutually 
independent random variables on a probability space 
$(\Omega, P)$ such that 
$P \{r_k=1\}= P \{r_k=-1\}=1/2$. 
We consider the following  
$\{c_{\veck}\}$:  
\[
c_{\veck} 
= r_{k_1+\cdots+k_N}(\omega) \prod_{j=1}^{N}
e^{-2\pi i|k_j|^{a_j}}. 
\]
Then we have 
\begin{align}
T_\sigma(f_{a_1, b_1}, \cdots, f_{a_N, b_N})(x)
&=
\Big\{ \calF^{-1} \varphi (x) \Big\}^{N}
\sum_{\veck \in D}
r_{k_1+\cdots+k_N}(\omega)
(1+ |\veck|)^{m}
\prod_{j=1}^{N}
|k_j|^{-b_j} 
e^{i x \cdot k_j}
\nonumber 
\\
&=
\Big\{ \calF^{-1} \varphi (x) \Big\}^{N}
\sum_{k \in \Z^n }
r_{k}(\omega) 
e^{i x \cdot k} 
d_k
\label{Tsfajbj}
\end{align}
with
\begin{equation*}
d_k
=
\sum_{ 
\substack{ \veck \in D, 
\\
k_1+ \cdots + k_N =k
 } }
(1+ |\veck|)^{m}
\prod_{j=1}^{N}
|k_j|^{-b_j}. 
\end{equation*}
Since 
$D$ is a finite set, 
$d_k=0$ except for a finite number of $k\in \Z^n$. 
From \eqref{Tsfajbj} and from 
our choice of $\varphi$, 
it follows that 
\begin{align}
\|
T_\sigma(f_{a_1, b_1}, \cdots, f_{a_N, b_N})
\|_{L^p(\R^n)}
&=
\left\|
\Big\{ \calF^{-1} \varphi (x) \Big\}^{N}
\sum_{ k\in \Z^n }
r_{k}(\omega) 
e^{i x \cdot k} 
d_k
\right\|_{L^p_x(\R^n)}
\nonumber 
\\
&\ge
\left\|
\sum_{ k \in \Z^n }
r_{k}(\omega) 
e^{i x \cdot k} 
d_k
\right\|_{L^p_x ([-\pi, \pi]^n)}. 
\label{lowerestimate}
\end{align}

Now combining 
\eqref{UniformboundTsigma}, 
\eqref{fjest}, and 
\eqref{lowerestimate}, 
we obtain 
\begin{align*}
\left\|
\sum_{k \in \Z^n }
r_{k}(\omega) 
e^{i x \cdot k} 
d_k
\right\|_{L^p_x ([-\pi, \pi]^n)}
\lesssim 1, 
\end{align*}
where 
the constant in 
$\lesssim$ does not depend on  $\omega \in \Omega$ 
and $D$. 
Raising to the power $p$ and 
averaging over $\omega \in \Omega $, 
we obtain 
\begin{align*}
\int_{[-\pi, \pi]^n}
\int_{\Omega}
\bigg|
\sum_{k} 
r_k(\omega) e^{ix \cdot k} d_k
\bigg|^p
dP(\omega) 
dx \lesssim 1. 
\end{align*}
By Khintchine's inequality, the above inequality is equivalent to 
\begin{equation}\label{dkell2}
\left(
\sum_{k\in \Z^n} 
|d_k|^2
\right)^{1/2}
\lesssim 1. 
\end{equation}
Notice that the implicit constant in \eqref{dkell2} 
does not depend on the 
finite set 
$D\subset (\Z^n)^N$.

We take the finite set $D\subset (\Z^n)^N$ defined as follows. 
We take a sufficiently large number $L>0$ 
and for sufficiently large $A \in \N$  
we set 
\begin{align*}
D=D_{A}
=
\Big\{
\veck \in (\Z^n)^N
\ :\ 
&
2^{A-1}\le |k_1 + \cdots + k_N| \le 2^{A+1} ,
\\&
2^{A-L-1} \le |k_j| \le 2^{A-L}, \quad j=1,\cdots,N-1
\Big\}. 
\end{align*}
If 
$L>0$ is chosen sufficiently large, 
then we have 
\[
\veck \in D_A 
\; \Rightarrow \; 
2^{A-2}\le |k_N| \le 2^{A+2}, 
\]
and thus 
for all $k\in \Z^n$ satisfying 
$2^{A-1}\le |k| \le 2^{A+1}$ we have 
\begin{align*}
d_k 
&
\approx 
\sum_{ \substack{ 
\veck \in D_A, 
\\
k_1+\cdots+k_N=k
} }
2^{A(m-b_1-\cdots-b_N)} 
\\
&
=
2^{A(m-b_1-\cdots-b_N)} 
\\
&
\quad 
\times 
\, \mathrm{card}\, 
\left\{(k_1, \dots, k_{N-1}) \in (\Z^n)^{N-1}
\mid 
2^{A-L-1} \le |k_j| \le 2^{A-L}, 
\; 
j=1, \dots, N-1
\right\}
\\
&
\approx
2^{A(m-b_1-\cdots-b_N)} \,
2^{An(N-1)}. 
\end{align*}
Thus 
\[
\bigg(
\sum_{k\in \Z^n} 
|d_k|^2
\bigg)^{1/2}
\approx
2^{A(m-b_1-\cdots-b_N)} 2^{An(N-1)} 2^{An/2} . 
\]
Hence 
\eqref{dkell2} implies 
\[
2^{A(m-b_1-\cdots-b_N)} 2^{An(N-1)} 2^{An/2} 
\lesssim 1. 
\]
Since this holds for arbitrarily large $A$, 
we have 
\begin{equation*}
m 
\le
\sum_{j=1}^{N} b_j -n(N-1) - \frac{n}{\, 2\, }. 
\end{equation*}
Taking limit as 
$a_j \to 0$ for $j \in J^{c}$, 
$a_j \to 1$ for $ j \in J$, 
and  
$\epsilon \to 0$, 
we obtain 
\begin{align*}
m 
&  
\le 
\sum_{j\in J^c} 
\bigg( n -   \frac{n}{\, p_j\, } \bigg)
+
\sum_{j\in J} 
\frac{n}{\, 2\, } 
-n(N-1) - \frac{n}{\, 2\, } 
\\
&
=
\frac{n}{\, 2\, } 
-
\sum_{j\in J^c} 
\frac{n}{\, p_j\, } 
-
\sum_{j\in J} 
\frac{n}{\, 2\, },  
\end{align*}
which is the inequality \eqref{criticalm}. 
\vs

%%=======================
{\it Case II: $2<p<\infty$ and $J\neq \emptyset$. } 
%%===============

In order to simplify notation, 
we assume 
$1\in J$, that is $2\le p_1 < \infty$. 
Suppose the multipliers 
$\sigma = \sigma (\xi_1, \xi_2, \dots, \xi_N)$ 
and 
$\tau=\tau (\xi_1, \xi_2, \dots, \xi_N)$ are related 
by 
\[
\tau (\eta, \xi_2, \dots, \xi_N)
=
\sigma (-\xi_2-\cdots - \xi_N - \eta , 
\xi_2, \dots, \xi_N). 
\]
Then 
\[
\int_{\R^n} T_{\sigma} (f_1, f_2, \cdots, f_N)(x) g(x) \,dx 
= \int_{\R^n} T_{\tau } (g,f_2, \cdots, f_{N})(x) f_1(x) \,dx 
\]
and thus duality 
implies 
\[
T_{\sigma}: 
L^{p_1} \times L^{p_2} \times \cdots \times L^{p_N}\to L^p
\; 
\Leftrightarrow\; 
T_{\tau} 
:
L^{p'} \times L^{p_2} \times \cdots \times L^{p_N}\to L^{p_1'}. 
\]
Also it is easy to see that 
\[
\sigma \in S_{0,0}^{m}(\R^n, N)^{\times} 
\; 
\Leftrightarrow\; 
\tau 
\in S_{0,0}^{m}(\R^n, N)^{\times}. 
\]
Hence 
\eqref{opxLp} implies 
\[
\op(S_{0,0}^{m}(\R^n,N)^{\times}) 
\subset
B(L^{p'} \times L^{p_2} \times \cdots \times L^{p_N} \to L^{p_1'}). 
\]
In the present case, 
we have 
$1 < p' <2 $ and 
$1 < p_1' \le 2$. 
Hence by what has been proved in Case I, 
we have 
\begin{align*}
m 
\le
\frac{\,n\,}{2} 
- \frac{n}{\,p'\, } 
-\sum_{j \in J \setminus \{1\}} \frac{n}{\,2\,} 
- \sum_{j\in J^c} 
\frac{n}{\,p_j\,} 
= 
\frac{\,n\,}{p} 
-\sum_{j \in J } \frac{n}{\,2\,} 
- \sum_{j\in J^c} \frac{n}{\,p_j\,}, 
\end{align*}
which is the inequality \eqref{criticalm}. 
\vs

%%=======================
{\it Case III: $2<p<\infty$ and $J= \emptyset$. } 
%%===============

In this case, 
$1<p_j<2$ for all $j$ and 
the condition \eqref{criticalm} reads as  
\begin{equation}\label{goal}
m 
\le
\frac{\,n\,}{p} - \sum_{j=1}^{N} 
\frac{n}{\,p_j\,}.  
\end{equation}

Take a function $\Psi \in C_{0}^{\infty} ((\R^n)^N)$ such that 
\begin{align*}
&
\supp \Psi \subset 
\left\{\vecxi \in (\R^n)^N \mid    2^{-1/2}N \le 
\textstyle{\sum_{j=1}^N } |\xi_j| \le 2^{1/2}N
\right\}, 
\\
&
\Psi (\vecxi) = 1 \;\; \text{if}\;\; 
2^{-1/4}N \le \textstyle{\sum_{j=1}^N } |\xi_j| \le 2^{1/4}N. 
\end{align*}
Let $m\in \R$ and 
consider the multiplier 
\[
\sigma (\vecxi) 
= 
\sum_{j=0}^{\infty} 2^{j m}
\Psi (2^{-j} \vecxi), 
\quad \vecxi \in (\R^n)^N, 
\] 
which certainly belongs to the class 
$S^{m}_{0,0}(\R^n, N)^{\times}$. 
Take a nonzero function $\psi \in C_{0}^{\infty} (\R^n)$ such that 
\[
\supp \psi \subset 
\{\xi \in \R^n \mid  2^{-1/4} \le |\xi| \le 2^{1/4}\}. 
\]  
Let $f_{j,k}\in \calS (\R^n)$ 
be 
the functions whose Fourier transforms are given by 
\[
\widehat{f_{j, k}}(\xi_j) = 2^{kn (1/p_j -1)} \psi (2^{-k} \xi_j), 
\quad j=1, \dots, N, \;\; k \in \N.  
\]
Then 
\[
\left\|
{f_{j, k}}
\right\|_{L^{p_j}} 
=
\left\| \calF^{-1} \psi \right\|_{L^{p_j}}
< \infty. 
\]
From the conditions on $\Psi$ and $\psi$, 
we have 
\begin{align*}
T_{\sigma} (f_{1,k}, \dots, f_{N,k} )(x) 
&= 
2^{km} \prod_{j=1}^{N} \calF^{-1} 
\left( 
2^{kn (1/p_j -1)} \psi (2^{-k} \xi_j)
\right)(x)
\\
&
=
2^{k(m + \sum_{j=1}^{N}  n/p_j ) }\left( 
\calF^{-1} \psi (2^k x )
\right)^{N}
\end{align*}
and hence 
\[
\left\|
T_{\sigma} (f_{1,k}, \dots, f_{N,k} )
\right\|_{L^p}
= c \, 
2^{k(m + \sum_{j=1}^{N} n/p_j - n/p)}
\]
with $c = \|(\calF^{-1}\psi)^N\|_{L^p} \in (0, \infty)$. 
Thus if \eqref{opxLp} holds we have 
$
2^{k(m + \sum_{j=1}^{N} n/p_j - n/p)}
\lesssim 1
$ 
for all $k\in \N$, 
which implies \eqref{goal}. 
This completes the proof of 
Lemma \ref{lem1-necessity}. 
\end{proof}

%%=============================
\begin{proof}[Proof of Theorem \ref{main-thm-necessity}] 
Let $N\ge 1$,  
$0 < p, p_1,\dots, p_N \le \infty$, and  
$1/p \le 1/p_1 + \cdots + 1/p_N$. 
We assume there exists an 
$\epsilon>0$ such that 
\begin{align}&
\op\left( S^{m}_{0,0}(\R^n, N)^{\times} \right) \subset
B(H^{p_1} \times \cdots \times H^{p_N} \to L^{p}), 
\label{assumption1}
\\&
m = 
\min \left\{ \frac{n}{\,p\,},\, \frac{n}{\,2\,} \right\} 
- \sum_{j =1}^{N} 
\max \left\{ 
\frac{n}{\,p_j\,},\, \frac{n}{\,2\,}  \right\} 
+ \epsilon, 
\label{assumption2}
\end{align}
with 
$L^{\infty}$ replaced by $BMO$ when 
$p=\infty$, and will derive a contradiction. 
In  
\cite[Theorem 6.1, Example 1.4]{KMT-arxiv-2} or 
in our Theorem \ref{main-thm-sufficiency} 
(or by Plancherel's theorem in the case $N=1$), 
it is already proved that 
\begin{equation}\label{factL2}
\op\left( S^{m}_{0,0}(\R^n, N)^{\times} \right) \subset
B(L^2 \times \cdots \times L^2 \to L^{2}), \quad
m=\frac{\,n\,}{2} - \frac{\,Nn\, }{2}. 
\end{equation}
From 
\eqref{assumption1}-\eqref{assumption2} 
and \eqref{factL2}, by 
complex interpolation, it follows that 
\begin{equation}\label{boundtildem}
\op\left( 
S^{\widetilde{m}}_{0,0}(\R^n, N)^{\times} \right) \subset
B(L^{ \widetilde{p_1} } \times \cdots \times L^{ \widetilde{p_N} } 
\to L^{ \widetilde{p} }),  
\end{equation}
where 
$\widetilde{p_j}, \widetilde{p}, \widetilde{m}$ are given by 
\begin{align*}
&
\frac{1}{ \widetilde{p_j} } 
= \frac{1-\theta}{2} + \frac{\theta}{p_j}, 
\quad j =1, \cdots, N, 
\\
&
\frac{1}{ \widetilde{p} } 
= \frac{\,1-\theta\,}{2} + \frac{\theta}{\,p\,}, 
\\
&
\widetilde{m}
= (1-\theta)
\left( \frac{\,n\,}{2} - \frac{\,Nn\, }{2} \right) 
+ \theta
\left( \min \left\{ \frac{n}{\,p\,},\, \frac{n}{\,2\,} \right\} 
- \sum_{j =1}^{N} 
\max \left\{ 
\frac{n}{\,p_j\,},\, \frac{n}{\,2\,}  \right\} 
+ \epsilon
\right)
\\
&\phantom{\widetilde{m}}
=\min \left\{ \frac{n}{\,\widetilde{p}\,},\, \frac{n}{\,2\,} \right\} 
- \sum_{j =1}^{N} 
\max \left\{ 
\frac{n}{\,\widetilde{p_j}\,},\, \frac{n}{\,2\,}  \right\} 
+ \theta \epsilon,   
\end{align*}
and 
$0 < \theta  < 1$ is a sufficiently small number. 
Notice that 
we have 
$1< \widetilde{p_j}<\infty$ 
for all $j$ 
if $\theta$ is 
sufficiently small. 
Hence 
\eqref{boundtildem} with the above 
$\widetilde{p_j}, \widetilde{p}, \widetilde{m}$ 
cannot hold as we have shown in 
Lemma \ref{lem1-necessity}. 
Thus Theorem \ref{main-thm-necessity} is proved. 
\end{proof}

%%%%%%%%%%%%%%%%%%%%%%%%%%%%%%%

%%=====================================


\begin{thebibliography}{20}

\bibitem{BBMNT}
\'A. B\'enyi, F. Bernicot, D. Maldonado, V. Naibo, and R. Torres,
{On the H\"ormander classes of bilinear pseudodifferential operators II},
Indiana Univ. Math. J. {\bf 62} (2013), 1733--1764.

\bibitem{BMNT}
\'A. B\'enyi, D. Maldonado, V. Naibo, and R. Torres,
{On the H\"ormander classes of bilinear pseudodifferential operators},
Integral Equations Operator Theory {\bf 67} (2010), 341--364.

\bibitem{BT-2003}
\'A. B\'enyi and R. Torres,
{Symbolic calculus and the transposes of
bilinear pseudodifferential operators},
Comm. PDE {\bf 28} (2003), 1161--1181.

\bibitem{BT-2004}
\'A. B\'enyi and R. Torres,
{Almost orthogonality and a class of bounded bilinear
pseudodifferential operators},
Math. Res. Lett. {\bf 11} (2004), 1--11.
	
\bibitem{CV}
A.P. Calder\'on and R. Vaillancourt,
{A class of bounded pseudo-differential operators},
Proc. Nat. Acad. Sci. U.S.A. {\bf 69} (1972), 1185--1187.

\bibitem{CM-Ast}
R.R. Coifman and Y. Meyer,
{Au del\`a des op\'erateurs pseudo-diff\'erentiels},
Ast\'erisque {\bf 57} (1978), 1--185.

\bibitem{CM-AIF}
R.R. Coifman and Y. Meyer,
Commutateurs d'int\'egrales singuli\`ers et op\'erateurs multilin\'eaires,
Ann. Inst. Fourier (Grenoble) {\bf 28} (1978), 177--202.

\bibitem{CKS-2015}
J. Cunanan, M. Kobayashi, and M. Sugimoto, 
Inclusion relations between $L^p$-Sobolev and Wiener amalgam spaces,
J. Funct. Anal. {\bf 268} (2015), 239--254.

\bibitem{F-1973}
C. Fefferman, 
$L^p$ bounds for pseudo-differential operators, 
Israel J. Math. {\bf 14} (1973), 413--417. 

\bibitem{Fei-1981}
	H.G. Feichtinger,
	Banach spaces of distributions of Wiener's type and interpolation,
	in: P.L. Butzer, B. Sz.-Nagy, and E. G\"orlich (eds),
	Functional Analysis and Approximation.
	ISNM 60: International Series of Numerical Mathematics,
	vol. 60, Birkh\"auser Basel, 1981, 153--165.
	
\bibitem{Fei-1983} 
	H.G. Feichtinger, 
	Modulation spaces on locally compact Abelian groups, 
	Technical report, University of Vienna, Vienna, 1983; 
	also in: M. Krishna, R. Radha, and S. Thangavelu (eds.),
	Wavelets and their applications, 
	(Allied, New Delhi, Mumbai, Kolkata, Chennai, Hagpur, 
	Ahmedabad, Bangalore, Hyderabad, Lucknow, 2003), 
	99--140.

\bibitem{fournier stewart 1985}
	J.J.F. Fournier and J. Stewart,
	Amalgams of $L^p$ and $\ell^q$,
	Bull. Amer. Math. Soc. (N.S.) {\bf 13} (1985), 1--21. 

\bibitem{GS-2004}
	Y.V. Galperin and S. Samarah,
	Time-frequency analysis on modulation spaces $M^{m}_{p,q}$, $0 < p,q \le \infty$,
	Appl. Comput. Harmon. Anal. {\bf 16} (2004), 1--18.

\bibitem{Garling}
  D. J. H. Garling, 
  {\it Inequalities, A Journey into Linear Analysis\/}, 
  Cambridge University Press, 
  2007. 
  	
\bibitem{goldberg 1979} 
	D. Goldberg, 
	A local version of real Hardy spaces, 
	Duke Math. J. {\bf 46} (1979), 27--42.

\bibitem{grafakos 2014c}
	L. Grafakos,
	{\it Classical Fourier analysis\/}, 
	3rd edition, GTM 249, Springer, New York, 2014.

\bibitem{grafakos 2014m}
	L. Grafakos,
	{\it Modern Fourier analysis\/}, 
	3rd edition, GTM 250, Springer, New York, 2014.

\bibitem{Gro-book-2001} 
	K. Gr\"ochenig, 
	{\it Foundation of time-frequency analysis\/},
	Birkh\"auser, Boston, 2001.
	
\bibitem{GCFZ-2019}
  W. Guo, J. Chen, D. Fan, and G. Zhao,
  Characterizations of Some Properties 
  on Weighted Modulation and Wiener Amalgam Spaces,
  Michigan Math. J. {\bf 68}, (2019), 451--482.

\bibitem{GWYZ-2017}
  W. Guo, H. Wu, Q. Yang, and G. Zhao,
  Characterization of inclusion relations 
  between Wiener amalgam and some classical spaces,
  J. Funct. Anal. {\bf 273}, (2017) 404--443.

\bibitem{holland 1975}
	F. Holland,
	Harmonic analysis on amalgams of $L^p$ and $\ell^q$, 
	J. London Math. Soc. (2) {\bf 10} (1975), 295--305.

\bibitem{KMT-arxiv}
  T. Kato, A. Miyachi, and N. Tomita,
  Boundedness of bilinear pseudo-differential operators of 
  $S_{0,0}$-type on $L^2 \times L^2$, 
  J. Pseudo-Differ. Oper. Appl. {\bf 12} Article number: 15 (2021).
%https://doi.org/10.1007/s11868-021-00391-1
%available at arXiv:1901.07237.

\bibitem{KMT-arxiv-2}
  T. Kato, A. Miyachi, and N. Tomita,
  Boundedness of multilinear pseudo-differential operators of 
  $S_{0,0}$-type in $L^2$-based amalgam spaces, 
  J. Math. Soc. Japan {\bf 73} (2021), 351--388. 
%available at arXiv:1908.11641.

\bibitem{Kob-2006} 
	M. Kobayashi, 
	Modulation spaces $M^{p,q}$ for $0 < p,q \leq \infty$, 
	J. Funct. Spaces Appl. {\bf{4}} (2006) 329--341.
		
\bibitem{MRS-2014}
  N. Michalowski, D. Rule, and W. Staubach,
  {Multilinear pseudodifferential operators beyond 
  Calder\'on-Zygmund theory}, 
  J. Math. Anal. Appl., \textbf{414}, (2014), 149--165. 

\bibitem{Miyachi-1980}
  A. Miyachi,
  {On some Fourier multiliers for $H^p (\R^n)$},
  J. Fac. Sci. Univ. Tokyo {\bf 27} (1980), 157--179.

\bibitem{Miyachi-1987}
  A. Miyachi,
  {Estimates for pseudo-differential operators of class 
  $S_{0,0}$},
  Math. Nachr. {\bf 133} (1987), 135--154.

\bibitem{MT-2013}
  A. Miyachi and N. Tomita,
  {Calder\'on-Vaillancourt--type theorem 
  for bilinear operators},
  Indiana Univ. Math. J. {\bf 62} (2013), 1165--1201.
	
\bibitem{PS-1988}
	L. P\"aiv\"arinta and E. Somersalo, 
	A generalization of the Calder\'on--Vaillancourt theorem to $L^p$ and $h^p$, 
	Math. Nachr. {\bf138} (1988), 145--156. 

\bibitem{Tri-1983}
	H. Triebel,
	Modulation spaces on the Euclidean $n$-space,
	Z. Anal. Anwendungen {\bf 2} (1983), 443--457. 
	
\bibitem{triebel 1983} 
	H. Triebel,
	{\it Theory of Function Spaces\/},
	Birkh\"auser, Verlag, 1983.

\bibitem{Wainger}
  S. Wainger,
  {Special Trigonometric Series in $k$-dimensions}, 
  Mem. Amer. Math. Soc. No. 59, 1965.

\bibitem{WH-2007} 
	B. Wang and H. Hudzik, 
	The global Cauchy problem for the NLS and NLKG with small rough data, 
	J. Differential Equations {\bf{232}} (2007) 36--73.
\end{thebibliography}
\end{document}